\documentclass[12pt]{amsart}

\usepackage[headings]{fullpage}
\usepackage{amsmath}
\usepackage{amsthm,amsfonts,amssymb,amscd}
\usepackage{lastpage}
\usepackage{enumerate}
\usepackage[mathscr]{eucal}
\usepackage{mathrsfs}
\usepackage[dvipsnames]{xcolor}
\usepackage{graphicx}
\usepackage{listings}
\usepackage{hyperref}
\usepackage{array}
\usepackage[utf8]{inputenc}
\usepackage[T1]{fontenc}
\usepackage{tikz}
\usepackage{tensor}
\usepackage{caption}
\usepackage{subcaption}
\usepackage[toc,page]{appendix}
\usepackage{datetime}
\usepackage{verbatim}
\usepackage{mathtools}
\usepackage[draft]{todonotes}

\def\Xint#1{\mathchoice
{\XXint\displaystyle\textstyle{#1}}%
{\XXint\textstyle\scriptstyle{#1}}%
{\XXint\scriptstyle\scriptscriptstyle{#1}}%
{\XXint\scriptscriptstyle\scriptscriptstyle{#1}}%
\!\int}
\def\XXint#1#2#3{{\setbox0=\hbox{$#1{#2#3}{\int}$ }
\vcenter{\hbox{$#2#3$ }}\kern-.6\wd0}}

\def\dashint{\Xint-}

\numberwithin{equation}{section}

\theoremstyle{plain}
\newtheorem{thm}{Theorem}[section]
\newtheorem{lemma}[thm]{Lemma}
\newtheorem{cor}[thm]{Corollary}
\newtheorem{prop}[thm]{Proposition}

\newtheorem{defn}[thm]{Definition}

\theoremstyle{remark}
\newtheorem{remark}{Remark}

\allowdisplaybreaks

\begin{document}

\author{Rodrigo Duarte} 
\author{Jorge Drumond Silva}
\address{Center for Mathematical Analysis, Geometry and Dynamical Systems,
	\newline
	Department of Mathematics, Instituto Superior T\'ecnico, Universidade de Lisboa
	\newline
	Av. Rovisco Pais, 1049-001 Lisboa, Portugal.}
\email{rodrigolealduarte@tecnico.ulisboa.pt}
\email{jsilva@math.tecnico.ulisboa.pt}

\thanks{R. Duarte was partially supported by FCT/Portugal through UID/MAT/04459/2020 and scholarship PD/BD/150338/2019.} 
\thanks{J. Drumond Silva was partially supported by FCT/Portugal through UID/MAT/04459/2020 and grant PTDC/MAT-PUR/1788/2020.}

\title{Weighted Gagliardo-Nirenberg Interpolation Inequalities}



\begin{abstract}

In this paper, we prove weighted versions of the Gagliardo-Nirenberg interpolation inequality with Riesz as well as Bessel type fractional derivatives, generalizing the celebrated result with classical derivatives. We use a harmonic analysis approach employing several methods, including the method of 
domination by sparse operators, to obtain such inequalities for a general class of weights satisfying
Muckenhoupt-type conditions. We also obtain improved results for some particular families of weights, including 
power-law weights $|x|^\alpha$. In particular, we prove an inequality which generalizes both the Stein-Weiss 
inequality and the Caffarelli-Kohn-Nirenberg inequality. However, our approach is sufficiently flexible to allow as well
for non-homogeneous weights and we also prove versions of the inequalities with Japanese bracket weights  $\langle x\rangle^\alpha=(1+|x|^2)^{\frac{\alpha}{2}}.$

\end{abstract}

\maketitle

\tableofcontents

\newcommand{\N}{\mathbb{N}}
\newcommand{\C}{\mathbb{C}}
\newcommand{\R}{\mathbb{R}}
\newcommand{\Z}{\mathbb{Z}}
\newcommand{\Q}{\mathbb{Q}}
\newcommand{\Primes}{\mathbb{P}}
\newcommand{\csigma}{\mathfrak{S}}
\newcommand{\1}{\mathbbm{1}}
\newcommand{\bb}[1]{\mathbb{#1}}
\newcommand{\mcal}[1]{\mathcal{#1}}
\newcommand{\T}{\mathbb{T}}
\newcommand{\A}{\mathbb{A}}
\newcommand{\di}{\text{div}}
\newcommand{\supp}{\text{supp}}
\newcommand{\I}{\mathcal{I}}
\newcommand{\D}{\mathcal{D}}
\newcommand{\Des}{\text{Des}}
\renewcommand{\S}{\mcal{S}}

\section{Introduction}

The Gagliardo-Nirenberg interpolation inequality is a fundamental result in the theory of Sobolev spaces, of extreme importance and usefulness
in partial differential equations and calculus of variations. It generalizes the classical Sobolev inequality, which is just one particular case of the Gagliardo-Nirenberg inequality at one of the end points of the range of values for the interpolation parameter.

First published in 1959, independently by Emilio Gagliardo \cite{Gagli} and Louis Nirenberg \cite{Nirenb}, with precursors in the works
of Hadamard, Nagy and Ladyzhenskaya, it has, due to its relevance,
become a significant object of research, with several variants and generalizations considered, in terms of domains, spaces of functions, etc. For a historic overview see \cite{Fiorenza} and references therein. 

The classical version of this inequality in $\R^d$ can be stated as follows.
\begin{thm}\label{thm1.1}
Let \(1\leq p,q,r\leq \infty\), $m\in \N_0$, $k\in \N$ and $\frac{m}{k}\leq \theta\leq 1$, be such that the following condition
\[
\frac{1}{r}-\frac{m}{d}=\theta\left(\frac{1}{p}-\frac{k}{d}\right)+(1-\theta)\frac{1}{q},
\]
holds. Then, we have that 
\[
\|\nabla^mf\|_{L^r(\R^d)}\lesssim_{m,k,d,p,q,\theta}\|\nabla^kf\|_{L^p(\R^d)}^\theta \|f\|_{L^q(\R^d)}^{1-\theta},\ \forall f\in L^q(\R^d)\cap \dot{W}^{k,p}(\R^d),
\]
with the exceptions that one must have $\theta<1$ if $r=\infty$ and $1<p<\infty$, and that if $q=\infty$, $k<d/p$ and $m=0$ then $f$ is assumed to
vanish at infinity.

\end{thm}

Obviously, the case $\theta=1$ corresponds to the Sobolev inequality, whereas the intermediate values of $\theta$ yield the truly interpolation cases. Actually, the complete statement of the Gagliardo-Nirenberg inequality extends to negative values of $r$, which should then be interpreted in terms of Hölder norms of classical derivatives, just like
for the Sobolev inequality when $k>d/p$. But that range of parameters will not be of interest for the results that we intend to prove here.
For a detailed proof of this classical version of the inequality see \cite{Fiorenza}, or \cite{Leoni}, for a modern textbook presentation. A classical reference, with a version of the inequality for bounded domains with smooth enough boundaries is \cite{Friedman}.

In this paper we study weighted versions of this inequality. Weighted $L^p$ norms frequently occur in partial differential equation estimates, so the need for a combination of the fundamental Gagliardo-Nirenberg inequality with weighted norms arises very naturally and is unquestionably an extremely useful result, often required as a technical
tool for applications in this field.

Weighted inequalities with power law weights were originally obtained by Caffarelli, Kohn 
and Nirenberg (\cite{Caffarelli}) who established an inequality able to generalize both 
the Gagliardo-Nirenberg inequality and Hardy's inequality. The Caffarelli-Kohn-Nirenberg inequality 
was then later improved by Chang Shou Lin, who proved the full Gagliardo-Nirenberg inequality
for integer derivatives and power law weights. We give the statement of Lin's result here for the purpose of comparison with 
some of our main results.
\begin{thm}[Lin \cite{Lin}]\label{thm1.4}
	Let \(1\leq p,q,r<\infty\), \(m\in \N_0, k\in \N\), \(\frac mk\leq \theta\leq 1\), $k-m - \frac dp \notin \N_0$ and $\alpha, \beta, \gamma \in \R$ are such that \(\alpha>-d/p, \beta>-d/q, \gamma>-d/r\) satisfying the conditions
	$$
	\frac{1}{r}-\frac{m-\gamma}{d}=\theta\left(\frac{1}{p}-\frac{k-\alpha}{d}\right)+(1-\theta)\left(\frac{1}{q}+\frac{\beta}{d}\right),$$
	and
	$$\ 0\leq \theta \alpha+(1-\theta)\beta -\gamma\leq \theta k-m.
	$$
	Then,
	\[
	\||x|^\gamma \nabla^m f\|_{L^r}\lesssim \||x|^\alpha \nabla^k f\|_{L^p}^\theta \||x|^\beta f\|_{L^q}^{1-\theta},\ \forall f\in C^\infty_c(\R^d).
	\]
\end{thm}

	Just as  for Theorem \ref{thm1.1}, the full statement of Lin's result in \cite{Lin} also includes the cases $r<0$, corresponding to Hölder regularity, as well as a slightly broader range of the parameters, that again we are not interested in considering here. The case $\alpha, \beta, \gamma=0$ obviously is
	the classical Gagliardo-Nirenberg inequality, while the case $k=1$, $m=0$  is  the Caffarelli-Kohn-Nirenberg inequality.

Other weighted Gagliardo-Nirenberg inequalities have been established involving different kinds of weights. For example, Duoandikoetxea and Vega, in \cite{DV}, studied conditions, as well as extremizers, for an absolutely arbitrary weight to guarantee that the corresponding weighted $L^2$  norm of a function is bounded by the product of the standard
$L^2$ norms of a function and its gradient. In some other works, different spaces were considered, like Orlicz spaces in \cite{KP} or BMO norms
in \cite{Le}.

In the particular framework of partial differential equations, there are several results in the literature about weighted versions
of the Sobolev and Gagliardo-Nirenberg inequalities adapted to this context. For example, Xavier Cabre and Xavier Ros-Oton, motivated by the study of regularity of stable solutions to reaction-diffusion problems in bounded 
domains with symmetry of double revolution (see \cite{Cabre}), established such
inequalities, with monomial weights defined on certain convex open cones of $\R^d$,
in \cite{Cabre2}.  Bonforte, Dolbeault, Muratori and Nazaret in \cite{Bonforte}
have obtained inequalities with power-law weights, motivated by the study of weighted fast diffusion 
equations. Giulio Ciraolo, Alessio Figalli and
Alberto Roncoroni were able to prove a weighted Sobolev inequality with homogeneous weights, which are not 
necessarily power-law weights or monomial weights (see \cite{Figalli}). The methods used rely on 
optimal mass transport theory, based on work by Cordero-Erausquin, Nazaret and Villani, who proved sharp 
Sobolev and Gagliardo-Nirenberg interpolation inequalities using this innovative method in \cite{Villani}.
The
inequalities obtained via this method have seen several generalizations, of which we highlight the recent work by
Zoltan Balogh, Segastiano Don and Alexandru Kristály, who were able to obtain in \cite{Balogh} a version 
of the Gagaliardo-Nirenberg inequality with different homogeneous weights on open convex cones of $\R^d$, satisfying  very specific 
concavity conditions.

 There are also results that involve 
general classes of weights which are not necessarily homogeneous. In \cite{Meyries}, for example, Meyries and Veraar obtain
Sobolev inequalities involving fractional derivatives and Muckenhoupt-type weights in the \(A_p\) class. This is also the framework that we follow in this paper, and our goal is to obtain the same type of result for the full Gagliardo-Nirenberg interpolation inequality, from a harmonic analysis point of view, specifically using the method of domination by sparse operators.

For our purposes it will be more convenient to work with Riesz fractional derivatives which are better suited to the Harmonic Analysis framework that we develop. Moreover, they generalize classical derivatives and the corresponding inequalities, in a particular sense which is explained in Remark \ref{remark on classical derivatives}, after Theorem \ref{thm1.5}. Therefore, given \(f\in \mcal{S}(\R^d)\)
and its Fourier transform
$$\hat{f}(\xi)=\int e^{-2\pi i x \cdot \xi}f(x)dx,$$
we define the operator as
\[
D^zf=(|\cdot|^z\hat{f})^\vee,
\]
which, in its fullest generality, is done by using the analytic continuation of $|\xi|^z$ as a tempered distribution valued holomorphic function of $z$, except for $z=-d-2k$, $k \in \N_0$, and interpreting this operator as a product of a tempered distribution by a Schwartz function in Fourier space, whose inverse Fourier transform corresponds then to a convolution of a tempered distribution (essentially $|x|^{-d-z}$) with $f$, thus yielding a $C^\infty$ function. But we will only be concerned with the range of parameters \(\mbox{Re}(z)>-d\), where the definition can simply be given
by the straightforward integral of the inverse Fourier transform, since in this case \(|\xi|^z\hat{f}\in L^1\).

This way our goal is to obtain inequalities of the form
\begin{equation}\label{maineq}
\|D^tf\|_{L^r(w)}\lesssim \|D^sf\|_{L^p(u)}^\theta \|f\|_{L^q(v)}^{1-\theta},
\end{equation}
where \(L^p(w)\) are weighted Lebesgue spaces with norm
\[
\|f\|_{L^p(w)}^p=\int_{\R^d}|f(x)|^pw(x)dx,
\]
with \(1\leq p<\infty\). 

We are also interested in the same inequality for the Bessel type (pseudodifferential) operator \(J^z\) defined by 
\[
J^zf=(\langle \cdot \rangle^z\hat{f})^\vee,\ f\in \mcal{S}'(\R^d),
\]
where $\langle \cdot \rangle^z=(1+|\cdot|^2)^{z/2}.$

Throughout the paper we use the Vinogradov notation
\[
A\lesssim B
\]
to mean that there is a constant \(C>0\) such that \(A\leq CB\). We sometimes use subscripts to denote the dependence of the constant on some parameters, however when this notation is used to control norms \(\|f\|_2\lesssim \|f\|_1\), it is understood that the implicit constant may depend on the parameters of the problem, but it does not depend on \(f\). Also, we consider only weight functions \(w\) that are locally integrable in \(\R^d\) and that satisfy \(0<w(x)<\infty\) for almost every \(x\in \R^d\). 

The project of obtaining weighted versions of important inequalities in harmonic analysis began in the 70s with the work of Muckenhoupt, who characterized the boundedness of the Hardy-Littlewood maximal operator in terms of the so called \(A_p\) condition (\cite{Muckenhoupt}). A weight \(w\) is said to satisfy the \(A_p\) condition, for \(1<p<\infty\), if
\[
\sup_Q\left(\dashint_Qw^{-\frac{p'}{p}}\right)^{p-1}\dashint_Q w<\infty,
\]
with $p'$ the conjugate exponent of $p$, i.e. $\frac 1p + \frac{1}{p'}=1$, and where the supremum is taken over all cubes \(Q\) with sides parallel to the coordinate axes. We use the notation
\[
\dashint_Qf=\frac{1}{|Q|}\int_Qf.
\]
In this paper we reserve the letter \(Q\) to denote cubes of the form
\[
Q=[x_1,x_1+l[\times \dots \times [x_d,x_d+l[,
\]
for some \(x=(x_1,\dots, x_d)\in \R^d\) and some side length \(l>0\). 

Soon after Muckenhoupt's result, many authors used the main ideas to prove the boundedness of several important operators in weighted spaces. See for instance \cite{CF}, \cite{HMW}, \cite{KW} and \cite{MW}. However, all of these results use only one weight. Moreover, the proofs rely on the reverse Hölder inequality which means that it is not easy to generalize these arguments to prove boundedness results between weighted spaces with different weights. Recently, a lot of progress was made in this direction by using the method of domination by sparse operators. The idea of this method is to control an operator by a so called sparse operator of the form
\[
\mcal{A}_{r,\mcal{S}}f(x)=\sum_{Q\in \mcal{S}}\chi_Q(x)\left(\dashint_Q|f|^r\right)^{1/r},
\]
where \(\mcal{S}\) is a suitably sparse family of cubes.
\begin{defn}\label{defn1.1}
We say that a collection of cubes \(\mcal{S}\) is \(\eta\)-sparse, \(0<\eta<1\), if for every \(Q\in \mcal{S}\) there is a measurable subset \(E(Q)\subseteq Q\) such that
\[
|E(Q)|\geq \eta |Q|,
\]
and moreover the sets \(\{E(Q)\}_{Q\in \mcal{S}}\) are pairwise disjoint.
\end{defn}
A surprising amount of operators can be controlled by sparse operators. The main result we will need in this paper is a recent theorem by Andrei Lerner and Sheldy Ombrosi.
\begin{thm}[Lerner, Ombrosi \cite{LernerOmbrosi}]\label{thm1.2}
Let \(T\) be a sublinear weak type \((q,q)\) operator with \(1\leq q<\infty\). Define the maximal operator \(M_T\) by
\[
M_Tf(x)=\sup_{Q\ni x}\sup_{x',x''\in Q}|T(f\chi_{\R^d\setminus Q^*})(x')-T(f\chi_{\R^d\setminus Q^*})(x'')|,
\]
where \(Q^*\) is the cube with the same center as \(Q\) but with side length \(5\sqrt{d}\) times the side length of \(Q\). Suppose that \(M_T\) is weak type \((r,r)\) with \(q\leq r<\infty\). Then, for every compactly supported function \(f\in L^r(\R^d)\) there exists a \((2\cdot 5^d\sqrt{d}^d)^{-1}\)-sparse family \(\mcal{S}=\mcal{S}(f)\) such that
\[
|Tf(x)|\lesssim_d(\|T\|_{L^q\rightarrow L^{q,\infty}}+\|M_T\|_{L^r\rightarrow L^{r,\infty}})\mcal{A}_{r,\mcal{S}}f(x)\text{ for a.e. }x\in \R^d.
\]
\end{thm}
The method of domination by sparse operators can be used to obtain two weight inequalities, essentially because it is much easier to control weighted norms of sparse operators. Of particular importance for our purpose is the work done on two weight inequalities for the Riesz potential, or fractional integral operator \(\I_\alpha\). Recall that this operator is defined by
\[
\I_\alpha f(x)=\int_{\R^d}\frac{f(x-y)}{|y|^{d-\alpha}}dy,
\]
where \(\alpha \in ]0,d[\), which is basically the same (except for multiplicative constants) as $D^{-\alpha}f$ (see \eqref{RieszDer} below), when $f \in \mcal{S}(\R^d)$. For a very good reference on many two weight results for the Riesz potential operator see \cite{Uribe}. A fundamental condition when dealing with two weight inequalities is the natural generalization of the \(A_p\) condition, which is the \(A_{p,q}\) condition.
\begin{defn}\label{defn1.2}
Let \(1<p,q<\infty\). We say the pair \((v,w)\) satisfies the \(A_{p,q}\) condition if 
\[
[v,w]_{A_{p,q}}:=\sup_Q\left(\dashint_Qv^{-\frac{p'} {p}}\right)^{1/p'}\left(\dashint_Qw\right)^{1/q}<\infty. 
\]
\end{defn}
This is a necessary condition for the two weight Hardy-Littlewood-Sobolev inequality to hold
\[
\|\I_\alpha f\|_{L^q(w)}\lesssim \|f\|_{L^p(v)},
\]
when
$$\frac{1}{q}+\frac{\alpha}{d}=\frac{1}{p}.$$
However, unlike in the one weight case, this turns out not to be a sufficient condition. To fix this, some extra condition is needed. There are many approaches to do this. For example, one can introduce a stronger condition than the \(A_{p,q}\) condition using Orlicz norms (\cite{Uribe},\cite{Perez}). However, for our purposes it will suffice to add the assumption that the weights $v^{-p'/p}$ and  $w$ satisfy the \(A_\infty\) condition. The reason for this is that we are mainly interested in applying these inequalities for weights like \(w(x)=|x|^\gamma, \gamma > -d,\) which are always in \(A_\infty\). 

Weighted versions of the Hardy-Littlewood-Sobolev inequality will be an important ingredient in the proof of the main results of this paper (see proposition \ref{weighted hardy-littlewood-sobolev}). 
We will start by using it to obtain a weighted Sobolev inequality of the type
\[
\|D^tf\|_{L^q(w)}\lesssim \|D^sf\|_{L^p(v)},
\]
with \(1<p<q<\infty\),\, \(t<s\) and the weights $v,w$ satisfying appropriate conditions (see theorem \ref{thm2.2}).
 
 Then, using ideas from complex interpolation and the method of domination by sparse operators we will be able to prove 
 the other end point of the weighted Gagliardo-Nirenberg inequality, from which follow very general versions of inequalities of the type
\eqref{maineq}, for three different Muckenhoupt type weights (see theorem \ref{thm2.5} below). Applying these general results for specific types of weights, we can recover
 Lin's result in Theorem \ref{thm1.4}  using power law weights, not only for integer derivatives but rather for the operators \(D\) and \(J\). Indeed we will prove the following theorem.
\begin{thm}\label{thm1.5}
Suppose that \(1<p,q,r<\infty, 0 \leq t<s, \frac ts \leq \theta \leq 1, \alpha\in]-d/p,d/p'[,\beta\in]-d/q,d/q'[,\gamma>-d/r\). Assume these satisfy
$$
    \frac{1}{r}-\frac{t-\gamma}{d}=\theta\left(\frac{1}{p}-\frac{s-\alpha}{d}\right)+(1-\theta)\left(\frac{1}{q}+\frac{\beta}{d}\right)
$$
and
$$
    0\leq \theta \alpha+(1-\theta)\beta-\gamma\leq \theta s-t.
$$
Then it follows that
\[
\||x|^\gamma D^tf\|_{L^r(\R^d)}\lesssim \||x|^\alpha D^sf\|_{L^p(\R^d)}^\theta\||x|^\beta f\|_{L^q(\R^d)}^{1-\theta},\ \forall f\in \mcal{S}(\R^d).
\]
Moreover, the same inequality is valid for the operator \(J\).
\end{thm}
\begin{remark}\label{remark on the size of gamma}
Under the conditions of the theorem it follows that \(\gamma < d/r'\). Indeed, using the fact that \(t-\theta s \leq 0\), \(\alpha < d/p'\) and \(\beta < d/q'\) we see that
\[
\begin{split}
\gamma &= d\left(\frac{\theta}{p}+\frac{1-\theta}{q}\right) + \theta \alpha + (1-\theta)\beta + t-\theta s -\frac{d}{r}\\
&< d\left(\frac{\theta}{p}+\frac{1-\theta}{q}\right) + d\left( \frac{\theta}{p'} + \frac{1-\theta}{q'}\right) - \frac{d}{r} = \frac{d}{r'}.
\end{split}
\]
In turn, this means that all the weights making an appearance are in the \(A_p\) class of the corresponding norm. More precisely, \(|x|^{\gamma r}\in A_r\), \(|x|^{\alpha p}\in A_p\) and \(|x|^{\beta q}\in A_q\). See section \ref{Homogeneous weights} for a characterization of the range of exponents for which the power-law weights belong to \(A_p\).  
\end{remark}
\begin{remark}\label{remark on classical derivatives}
We note also that when \(s, t\) are positive integers then the norms are equivalent to the norms with classical derivatives. More precisely, if \(f\in \mcal{S}(\R^d), m\in \N_1\) and \(-d/p < \gamma < d/p'\) then
\[
\||x|^\gamma D^m f\|_{L^p(\R^d)} \sim \||x|^\gamma \nabla^m f\|_{L^p(\R^d)}.
\]
Indeed, given a multi-index \(\alpha\) with \(|\alpha|=m\), define the multiplier operator \(T_\alpha f = (\xi^\alpha |\xi|^{-m}\hat{f})^\vee\). From \cite{KW}, we know that \(T_\alpha\) is bounded on \(L^p(w)\) when \(w\in A_p\). In the present case this means that 
\[
\||x|^\gamma T_\alpha g\|_{L^p(\R^d)}\lesssim \||x|^\gamma g\|_{L^p(\R^d)}.
\]
Now, we can relate \(D^mf\) to \(\nabla^mf\) by using \(T_\alpha\). In fact, there are constants \(c_\alpha, c_\alpha '\) such that
\[
D^mf = \sum_{|\alpha|=m}c_\alpha T_\alpha (\partial^\alpha f)\text{ and conversely }\partial^\alpha f = c'_\alpha T_\alpha D^mf.
\]
Therefore,
\[
\||x|^\gamma D^mf\|_{L^p(\R^d)}\lesssim \sum_{|\alpha|=m}\||x|^\gamma T_\alpha \partial^\alpha f\|_{L^p}\lesssim \sum_{|\alpha|=m}\||x|^\gamma \partial^\alpha f\|_{L^p}.
\]
And conversely,
\[
\||x|^\gamma \nabla^mf\|_{L^p(\R^d)}\lesssim \sum_{|\alpha|=m}\||x|^\gamma T_\alpha D^mf\|_{L^p}\lesssim \||x|^\gamma D^mf\|_{L^p}.
\]
Therefore the classical inequality is indeed a particular case of theorem \ref{thm1.5}. 
\end{remark}
Our method is much more flexible than the classical methods used by Lin since they don't rely on the homogeneity of the weights. In fact we can easily prove analogous results for other types of weights. Of particular interest also are the non-homogeneous weights \(w(x)=\langle x\rangle^\gamma\).
\begin{thm}\label{thm1.6}
Suppose that \(1<p,q,r<\infty,\, 0\leq t<s,\,\, t/s \leq \theta \leq 1, \alpha\in]-d/p,d/p'[,\beta\in]-d/q,d/q'[,\gamma>-d/r\). Assume these satisfy
\begin{align*}
    \frac{1}{r}&\leq \frac{\theta}{p} + \frac{1-\theta}{q}, \\
    \gamma &\leq \theta \alpha+(1-\theta)\beta
\end{align*}
and
\[
\theta\left(\frac{1}{p}-\frac{s}{d}\right)+\frac{1-\theta}{q}\leq \frac{1}{r}-\frac{t}{d}\leq \theta\left(\frac{1}{p}-\frac{s-\alpha}{d}\right)+(1-\theta)\left(\frac{1}{q}+\frac{\beta}{d}\right)-\frac{\gamma}{d}.
\]
Then it follows that
\[
\|\langle x\rangle^\gamma D^tf\|_{L^r(\R^d)}\lesssim\|\langle x\rangle^\alpha D^sf\|_{L^p(\R^d)}^\theta\|\langle x\rangle^\beta f\|_{L^q(\R^d)}^{1-\theta},\ \forall f\in \mcal{S}(\R^d).
\]
Moreover, the same result is valid for the operator \(J\).
\end{thm}
The proofs of these results all follow the same strategy. The idea is to consider first the extreme cases \(\theta=1\) and \(\theta=t/s\). The reason for this is that if the inequalities are satisfied at both of these extremes then we can apply the \(\theta=1\) case, a Sobolev inequality, to get
\[
\|D^tf\|_{L^r}\lesssim \|D^{\theta s}f\|_{L^a}
\]
with \(1/a=\theta/p+(1-\theta)/q\) and where the value of $ t/s< \theta < 1$ here is now the one that we seek for the final Gagliardo-Nirenberg interpolated inequality. Then, we can use the other extreme inequality to get
\[
\|D^{\theta s}f\|_{L^a}\lesssim \|D^sf\|_{L^p}^\theta \|f\|_{L^q}^{1-\theta},
\]
again with $ t/s< \theta < 1$ for the final interpolated value.
The first inequality works because
\[
\frac{1}{r}-\frac{t}{d}=\frac{1}{a}-\frac{\theta s}{d}.
\]
And the second one works because
\[
\frac{1}{a}=\frac{\theta}{p}+\frac{1-\theta}{q}.
\]
By using this idea, the problem is essentially reduced to proving a weighted Sobolev inequality (\(\theta=1\)) and the minimum exponent inequality (\(\theta=t/s\)). For the Sobolev inequality, the idea is to compare with the Riesz potential operator and reduce the problem to proving the weighted Hardy-Littlewood-Sobolev inequality
\[
\|\I_\alpha f\|_{L^q(w)}\lesssim \|f\|_{L^p(v)}.
\]
As for the minimum exponent inequality, the idea is to use results from interpolation theory, namely the three lines lemma. As we will see later, for this approach to work we need an estimate of the form
\[
\|D^{i\tau}f\|_{L^p(w)}\lesssim_\tau \|f\|_{L^p(v)},
\]
and we need to be careful about the explicit dependence of the constant on \(\tau\). To prove this inequality we use theorem \ref{thm1.2} and so we must check that this theorem can be applied to the operators \(D^{i\tau}\) and \(J^{i\tau}\).

The paper is structured as follows. In section \ref{general} we consider this approach with general weights, separating between the cases \(\theta=1\) and \(\theta=t/s\). In section \ref{particular} we study particular families of weights and our main goal then will be to prove theorems \ref{thm1.5} and \ref{thm1.6}. In section \ref{nonhomderivatives} we show how to adapt the arguments to non-homogeneous derivatives and finally in section \ref{mixed} we give a simple application to a mixed norm inequality.

\section{General weights}\label{general}

In this section we consider general weight functions \(w\) assuming only that they are locally integrable and \(0<w(x)<+\infty\) almost everywhere. Our goal is to obtain an inequality of the form
\[
\|D^tf\|_{L^r(w)}\lesssim \|D^sf\|_{L^p(u)}^\theta\|f\|_{L^q(v)}^{1-\theta}.
\]

As was explained in the introduction, we first consider the cases \(\theta=1\) and \(\theta=t/s\). We begin with the case \(\theta=1\) which is a weighted version of the homogeneous Sobolev inequality.

\subsection{Weighted homogeneous Sobolev inequality}
The goal of this section is to prove the inequality
\[
\|D^tf\|_{L^q(w)}\lesssim \|D^sf\|_{L^p(v)},
\]
with $1<p<q<\infty$. Intuitively, writing $D^{t}f=D^{t-s}D^{s}f$  and setting \(\alpha=s-t\) we are led to the weighted Hardy-Littlewood-Sobolev inequality
\begin{equation}\label{eq2.2}
\|\I_\alpha \phi\|_{L^q(w)}\lesssim \|\phi\|_{L^p(v)},
\end{equation}
for $\phi= D^s f$. 

To prove inequality \eqref{eq2.2} the idea is to dominate \(\I_\alpha\) by the sparse operator 
\[
\mcal{A}_\S^\alpha f(x) = \sum_{Q\in \S}\chi_Q(x)|Q|^\frac{\alpha}{d}\dashint_Q|f|.
\]
From \cite{Uribe} we can compare \(\I_\alpha\) pointwise to this operator. Indeed, the following two propositions hold.
\begin{prop}[\cite{Uribe}, proposition 3.6]\label{prop3.1}
There are dyadic lattices \(\mcal{D}^1,\dots, \mcal{D}^{3^d}\) such that for any \(f\geq 0\) measurable and bounded with compact support, we have that
\[
\mcal{A}^\alpha_{\D^j} f(x)\lesssim_{d,\alpha} \I_\alpha f(x)\lesssim_{d,\alpha} \sum_{j=1}^{3^d}\mcal{A}^\alpha_{\D^j}f(x),\ \forall x\in \R^d.
\]
\end{prop}
\begin{prop}[\cite{Uribe}, proposition 3.9]\label{prop3.2}
Let \(\D\) be a dyadic lattice, \(\alpha\in]0,d[\) and let \(f\geq 0\) be a bounded, measurable function with compact support. Then, there exists a sparse set \(\mcal{S}=\mcal{S}(f)\subseteq \D\) such that
\[
\mcal{A}^\alpha_\mcal{S}f(x)\leq \mcal{A}^\alpha _\D f(x)\lesssim_{d,\alpha} \mcal{A}^\alpha_\mcal{S}f(x),\ \forall x\in \R^d.
\]
\end{prop}
This reduces the estimate to proving two-weight norm inequalities for the sparse operator \(\mcal{A}_\S ^\alpha\). These types of weighted inequalities naturally lead to the consideration of a more general condition than the \(A_{p, q}\) condition.
\begin{defn}
Let \(1<p,q<\infty\) and \(\alpha\in\R\). We say two weights \(v,w\) satisfy the \(A_{p,q}^\alpha\) condition if 
\[
[v,w]_{A_{p,q}^\alpha}:=\sup_Q|Q|^{\frac{\alpha}{d}-1}\left(\int_Qv^{-\frac{p'}{p}}\right)^{1/p'}\left(\int_Qw\right)^{1/q}<\infty.
\]
The set of all such weights is denoted by \(A_{p,q}^\alpha\).
\end{defn}
The problem of obtaining two-weight norm inequalities for sparse operators has been extensively studied, and for our purposes, theorem 1.1 from \cite{Hytonen bounds for sparse operators} will suffice.
\begin{thm}[Theorem 1.1 from \cite{Hytonen bounds for sparse operators} with \(r=1\)]\label{bound for sparse operators}
Let \(1 < p \leq q < \infty\), \(0\leq \alpha < d\) and let \(\S\) be a sparse family of cubes from a dyadic lattice \(\mcal{D}\). If \(v^{-\frac{p'}{p}}, w\in A_\infty\) and \((v, w) \in A_{p,q}^\alpha\), then 
\[
\|\mcal{A}_\S^\alpha f\|_{L^q(w)}\lesssim \|f\|_{L^p(v)}.
\]
\end{thm}

We now use the sparse operator to control the Riesz potential operator and obtain a weighted Hardy-Littlewood-Sobolev inequality.
\begin{prop}\label{weighted hardy-littlewood-sobolev}
Let \(1<p\leq q<\infty, 0<\alpha<d\). Suppose that \(v, w\) are weights such that \((v, w)\in A_{p,q}^\alpha\) and \(v^{-\frac{p'}{p}}, w\in A_\infty\). Then, we have that
\[
\| \I_\alpha f\|_{L^q(w)}\lesssim \| f\|_{L^p(v)}.
\]
\end{prop}
\begin{proof}
Without loss of generality assume that \(f\geq 0\). Now put \(f_N=\min\{f,N\}\chi_{B_N}\). Then, we know that 
\[
\I_\alpha(f_N)(x)\lesssim_{d,\alpha}\sum_{j=1}^{3^d}\mcal{A}_{\mcal{S}^j}^\alpha (f_N)(x),
\]
where \(\mcal{S}^j\subseteq \D^j\) are sparse sets. Therefore, from theorem \ref{bound for sparse operators}, we get
\[
\| \I_\alpha (f_N)\|_{L^q(w)}\lesssim\sum_{j=1}^{3^d}\| \mcal{A}_{\mcal{S}^j}^\alpha(f_N)\|_{L^q(w)}\lesssim\sum_{j=1}^{3^d}\| f_N\|_{L^p(w)}\lesssim\| f\|_{L^p(w)}.
\]
The result now follows from the monotone convergence theorem.
\end{proof}

Now that we have a weighted Hardy-Littlewood-Sobolev inequality at our disposal, we simply have to relate this to the Riesz fractional derivatives. For $0<\mbox{Re}(\alpha)<d$, if \(\phi\in \mcal{S}(\R^d)\), then 
\begin{equation}\label{RieszDer}
	\begin{split}
		D^{-\alpha}\phi(x)&=\int_{\R^d}e^{2\pi i \xi\cdot x}|\xi|^{-\alpha}\hat{\phi}(\xi)d\xi = c_{\alpha,d}\int_{\R^d}|y|^{\alpha-d}\phi(x-y)dy\\
		&=c_{\alpha,d}\I_\alpha \phi(x),
	\end{split}
\end{equation}
where we used the well-known\footnote{For a reference see \cite{CGrafakos}, section 2.4.3}
fact that \((|\cdot|^{-\alpha})^\wedge(\xi)=c_{\alpha,d}|\xi|^{\alpha-d}\) and
\[
c_{\alpha,d}=\pi^{\alpha-\frac{d}{2}}\frac{\Gamma\left(\frac{d-\alpha}{2}\right)}{\Gamma\left(\frac{\alpha}{2}\right)}.
\]

This argument works particularly well if one considers only Schwartz functions with zero moments:
\[
\mcal{S}_0(\R^d):=\{f\in \mcal{S}(\R^d): \int_{\R^d}x^\gamma f(x)dx=0\text{ for all multi-indices }\gamma\}. 
\]
These functions have the nice property that \(D^z\) maps \(\mcal{S}_0(\R^d)\) into itself homeomorphically, where the topology is the induced topology from \(\mcal{S}(\R^d)\) and \(z\) is any complex number. Also the semigroup property \(D^{z+w}=D^zD^w\) is satisfied in this space. 
This relationship allows us to use theorem \ref{weighted hardy-littlewood-sobolev} immediately to obtain, for \(f\in \mcal{S}_0(\R^d)\),
\[
\|D^tf\|_{L^q(w)}=\|D^{-\alpha}D^sf\|_{L^q(w)}=c_{\alpha,d}\|\I_\alpha[D^sf]\|_{L^q(w)}\lesssim \|D^sf\|_{L^p(v)}.
\]
However, we would also like to obtain such an inequality for Schwartz functions. One reason why proving this only for \(\mcal{S}_0(\R^d)\) is unsatisfying is that this space is not dense in weighted \(L^p\) spaces where Schwartz functions are dense.\footnote{For example, \(\S(\R^d)\) is dense in \(L^2(\langle x\rangle^{d+1})\), but \(\mcal{S}_0(\R^d)\) is not.} However, extending this to Schwartz functions presents some issues because the semigroup property of the composition of Riesz derivatives does not work so well, since $D^z$ does not map $\mcal{S}(\R^d)$ to $\mcal{S}(\R^d)$. With a bit more care, though, it is nevertheless still possible to do this. 

We begin by defining the spaces
\[
D^{\geq a}\mcal{S}(\R^d)=\{f\in L^\infty: f=D^bg\text{ for some }g\in \mcal{S}(\R^d)\text{ and }b\geq a\},\ a>-d.
\]
We could similarly define \(D^{>a}\mcal{S}(\R^d)\) and \(D^{=a}\mcal{S}(\R^d)\). One easily extends the definition of the operator \(D^s\) to these spaces by setting, for example, for $f\in D^{=a}\mcal{S}(\R^d)$,
\[
D^sf(x)=\int_{\R^d}e^{2\pi i x\cdot \xi}|\xi|^s\hat{f}(\xi)d\xi=\int_{\R^d}e^{2\pi i x\cdot \xi}|\xi|^{s}|\xi|^{a}\hat{g}(\xi)d\xi,
\]
as long as \(a>-d\) and \(a+s>-d\). Functions in \(D^{\geq 0}\mcal{S}(\R^d)\) are especially well behaved. Indeed it is clear that if \(f\in D^{\geq 0}\mcal{S}(\R^d)\), then \(\langle \xi\rangle^s\hat{f}=\langle \xi\rangle^s|\xi|^b\hat{g} \in L^2\) for any \(s\in \R\), so \(D^{\geq 0}\mcal{S}(\R^d)\subseteq H^s(\R^d),\forall s\in \R\), where here \(H^s\) is the Sobolev space
\[
H^s(\R^d)=\{f\in \mcal{S}'(\R^d): \langle \xi\rangle^s\hat{f}\in L^2\}. 
\]
By the Sobolev embedding theorem we can immediately conclude that functions in \(D^{\geq 0}\mcal{S}(\R^d)\) are smooth, and all their partial derivatives of any order (including the function itself) are bounded, square-integrable, Hölder continuous of any exponent \(0<\delta<1\) and go to zero at infinity. Moreover, for functions in \(D^{>0}\mcal{S}(\R^d)\) we can prove a quantitative decay at infinity.
\begin{prop}\label{prop2.1}
If \(g\in \mcal{S}(\R^d)\) and \(s>0\), then \(|x|^sD^sg(x)\in L^\infty\).
\end{prop}
To prove this we use the following lemma.
\begin{lemma}\label{extralemma1}
Suppose \(f\in C^N(\R^d), N\in \N_0\) with \(|x|^{d+1}\partial^\gamma f\in L^\infty(\R^d), \forall |\gamma|\leq N\). Moreover, suppose that for all \(\gamma\) with \(|\gamma|=N\) there is a constant \(C_\gamma\) such that
\begin{equation}\label{eq2.3}
|\partial^\gamma f(x)-\partial^\gamma f(y)|\leq C_\gamma \frac{|x-y|^\delta}{(1+\min\{|x|,|y|\})^{d+1}},
\end{equation}
for some \(0<\delta \leq 1\). Then, it follows that \(|\xi|^{N+\delta}\hat{f}\in L^\infty(\R^d)\).
\end{lemma}
The proof is simple so we omit it from the main text, but for completeness we have included it in appendix \ref{appendixA}.
\begin{proof}[Proof of proposition \ref{prop2.1}]
By lemma \ref{extralemma1} it suffices to check that \(f(\xi)=|\xi|^s\hat{g}(\xi)\) is in \(C^N(\R^d)\), \(|\xi|^{d+1}\partial^\gamma f\in L^\infty(\R^d), \forall |\gamma|\leq N\) and we have the estimate
\[
|\partial^\gamma f(\xi)-\partial^\gamma f(\eta)|\leq C_\gamma \frac{|\xi-\eta|^\delta}{(1+\min\{|\xi|,|\eta|\})^{d+1}}, |\gamma|=N,
\]
with \(N=[s]\) and \(\delta=s-[s]\).\footnote{Here \([s]\) denotes the greatest integer that is strictly smaller than \(s\).} That \(f\) is in \(C^N(\R^d)\) is straightforward to check. One can also prove by induction the formula for the derivatives
\[
\partial^\gamma f(\xi)=\sum_{\gamma_1\leq \gamma}\sum_{\gamma_2\leq \gamma_1}C_{\gamma_1,\gamma_2,s}\xi^{\gamma_2}|\xi|^{s-|\gamma_1|-|\gamma_2|}\partial^{\gamma-\gamma_1}\hat{g}.
\]
Since \(g\in \mcal{S}(\R^d)\), it follows that \(|\xi|^{d+1}|\partial^\gamma f(\xi)|\lesssim\sum_{\gamma_1\leq \gamma} |\xi|^{d+1}|\xi|^{s-|\gamma_1|}|\partial^{\gamma-\gamma_1}\hat{g}(\xi)|\in L^\infty\), since \(d+1+s-|\gamma_1|\geq 0\). Finally, it suffices to check that each term of \(\partial^\gamma f,|\gamma|=N\), satisfies estimate \eqref{eq2.3}. This is done in appendix \ref{appendixA}, lemma \ref{lemmaA.3}.
\end{proof}

Now let \(s,t\) be such that \(t<s<t+d\). Suppose that \(f\in D^{>\tau}\mcal{S}(\R^d)\) where \(\tau=\max\{-t,-d\}\). Then we may write \(f=D^ag\) for some \(g\in \mcal{S}(\R^d)\) and some \(a>\tau\). Set also \(\alpha=s-t\in ]0,d[\). Under these assumptions and using proposition \ref{prop2.1} we see that
\[
\begin{split}
\I_\alpha(|D^sf|)(x)&= \int_{|y|\leq 2|x|}|y|^{\alpha-d}|D^{s+a}g(x-y)|dy+\int_{|y|>2|x|}|y|^{\alpha-d}|D^{s+a}g(x-y)|dy\\
&\lesssim\int_0^{2|x|}\rho^{\alpha-d}d|B_1|\rho^{d-1}d\rho+ \int_{|y|>2|x|}|y|^{\alpha-d}|x-y|^{-s-a}dy\\
&\lesssim |x|^\alpha+\int_{|y|>2|x|}|y|^{\alpha-d}|y|^{-s-a}dy\\
&\lesssim |x|^\alpha+|x|^{-t-a}.
\end{split}
\]
We use this estimate if \(|x|>1\). If \(|x|\leq 1\) we can split between the regions \(|y|\leq 2\) and \(|y|>2\), and we find that \(\I_\alpha(|D^sf|)\) is bounded in \(\overline{B}_1\). Therefore, we have the estimate
\begin{equation}\label{eq2.4}
\I_\alpha(|D^sf|)(x)\lesssim \langle x\rangle^\alpha.
\end{equation}
Now suppose that \(\phi\in \mcal{S}_0(\R^d)\). Then we have that
\[
\langle \I_\alpha(D^sf), \phi\rangle=\int \I_\alpha(D^sf)(x)\phi(x)dx=\int \phi(x)\int |y|^{\alpha-d}D^sf(x-y)dydx.
\]
Using \eqref{eq2.4} we see that
\[
\int\int |\phi(x)||y|^{\alpha-d}|D^sf(x-y)|dydx=\int |\phi(x)|\I_\alpha(|D^sf|)(x)dx\lesssim \int |\phi(x)|\langle x\rangle^\alpha dx<\infty.
\]
Repeatedly applying Fubini and noting that \(\hat{g}|\cdot|^{s+a}\check{\phi}\in \mcal{S}(\R^d)\) we obtain
\[
\begin{split}
\langle \I_\alpha(D^sf),\phi\rangle &=\int |y|^{\alpha-d}\int \phi(x)D^sf(x-y)dxdy\\
&=\int |y|^{\alpha-d}\int \phi(x)\int e^{2\pi i (x-y)\cdot \xi}|\xi|^{s+a}\hat{g}(\xi)d\xi dxdy\\
&=\int |y|^{\alpha-d}\int e^{-2\pi i y\cdot \xi}|\xi|^{s+a}\hat{g}(\xi)\int e^{2\pi i x\cdot \xi}\phi(x)dxd\xi dy\\
&=\int |y|^{\alpha-d}(|\cdot|^{s+a}\hat{g}\check{\phi})^\wedge(y)dy\\
&=\frac{1}{c_{\alpha,d}}\int |\xi|^{-\alpha}|\xi|^{s+a}\hat{g}(-\xi)\check{\phi}(-\xi)d\xi =\frac{1}{c_{\alpha,d}}\int \phi(x)\int e^{2\pi i x\cdot \xi} |\xi|^{t+a}\hat{g}(\xi)d\xi dx\\
&=\frac{1}{c_{\alpha,d}}\int \phi(x)D^tf(x)dx=\frac{1}{c_{\alpha,d}}\langle D^tf,\phi\rangle,
\end{split}
\]
where again we used the fact that \((|\cdot|^{-\alpha})^\wedge(\xi)=c_{\alpha,d}|\xi|^{\alpha-d}\). Since this holds for all \(\phi\in \mcal{S}_0(\R^d)\), then for any \(\psi\in \mcal{S}(\R^d)\) that is zero in a neighborhood of the origin we have \(\langle (c_{\alpha,d}\I_\alpha(D^sf)-D^tf)^\wedge,\psi\rangle=0\). In other words, \(\supp((c_{\alpha,d}\I_\alpha(D^sf)-D^tf)^\wedge)\subseteq \{0\}\). This means that there is some polynomial \(P\) such that \(c_{\alpha,d}\I_\alpha(D^sf)=D^tf-P\). Given that \(D^tf\) vanishes at infinity\footnote{We say a function \(f\) vanishes at infinity with respect to the measure \(\mu\) if \(\mu(\{x: |f(x)|>\lambda\})<\infty,\forall \lambda>0\). We know that \(D^tf\) vanishes at infinity with respect to the Lebesgue measure by the Riemann-Lebesgue lemma.} and \(\I_\alpha(D^sf)\) grows at most like \(\langle x\rangle^\alpha\), it follows that \(\text{deg}(P)\leq \lfloor \alpha\rfloor\). Using theorem \ref{weighted hardy-littlewood-sobolev} we see that
\[
\|D^tf-P\|_{L^q(w)}=c_{\alpha,d}\|\I_\alpha(D^sf)\|_{L^q(w)}\lesssim \|D^sf\|_{L^p(v)}.
\]
This establishes the main result of this section.
\begin{thm}\label{thm2.2}
Let \(1<p\leq q<\infty\) and \(t<s<t+d\). Suppose \((v,w)\in A_{p,q}^{s-t}\), \(v^{-p'/p}, w \in A_\infty\) and set \(\tau=\max\{-d,-t\}\). Then, for any \(f\in D^{>\tau}\mcal{S}(\R^d)\) there exists a polynomial \(P\) of degree at most \(\lfloor s-t \rfloor\) such that
\begin{equation}\label{eq2.5}
\|D^tf-P\|_{L^q(w)}\lesssim\|D^sf\|_{L^p(v)},
\end{equation}
where the implicit constant depends on \(v,w,p,q,d,s\) and \(t\), but not on \(f\).
\end{thm}
\begin{remark}\label{remark1}
If \(t>0\) then \(\tau<0\) and so \(D^{>\tau}\mcal{S}(\R^d)\supseteq \mcal{S}(\R^d)\). This means that when \(t>0\) this inequality holds for Schwartz functions. If \(t=0\) then we just barely miss the Schwartz class, but we can recover it if we make an additional assumption on the weights. Indeed, suppose \(v\) is such that \(L^p(v)\subseteq L^p\). Then, if \(D^sf\in L^p(v)\) it would follow that \(\I_\alpha(|D^sf|)\in L^q\) by the boundedness of the Riesz potential operator. In this case we see that
\[
\int \int |\phi(x)||y|^{\alpha-d}|D^sf(x-y)|dydx=\int |\phi(x)|\I_\alpha(|D^sf|)(x)dx\leq \|\phi\|_{L^{q'}}\|\I_\alpha(|D^sf|)\|_{L^q}<\infty,
\]
and therefore all the computations could be carried out the same way to yield inequality \eqref{eq2.5}. Alternatively, we could use the assumption \(\mcal{S}(\R^d)\subseteq L^{q'}(w^{-q'/q})\), for in this case \(\I_\alpha(|D^sf|)\in L^q(w)\) and \(\phi\in L^{q'}(w^{-q'/q})\).
\end{remark}
Note that \(D^sf\in L^2\), so if \(\alpha<d/2\), then by the boundedness of the Riesz potential operator it would follow that \(\I_\alpha(D^sf)\in L^q\), where
\[
\frac{1}{q}+\frac{\alpha}{d}=\frac{1}{2}.
\]
In this case we could conclude that \(P=D^tf-c_{\alpha,d}\I_\alpha(D^sf)\) vanishes at infinity and is therefore zero. However, \(\alpha=d(p^{-1}-q^{-1})\), which can be arbitrarily close to \(d\), and thus in general we cannot use this argument to say that \(P\) must be zero. Under the additional assumption that \(D^sf\in L^p(v)\), we would get \(\I_\alpha(D^sf)\in L^q(w)\), but for general weights this still does not imply that \(P\) must vanish at infinity. Thus, to eliminate this polynomial we need some extra assumption on the weight \(w\). The idea is that \(w\) cannot decay so fast at infinity that a non-zero polynomial could nevertheless still vanish at infinity with respect to the measure \(w(x)dx\). In the next lemma we give a condition that ensures this cannot happen.
\begin{lemma}\label{lemma2.4}
Given \(R>0\) and \(S\subseteq S^{d-1}\) a subset of positive (standard surface) measure of the $d-1$ dimensional unit sphere, define the cone
\[
\Gamma_R^S:=\{x\in \R^d: |x|\geq R\text{ and }\frac{x}{|x|}\in S\}. 
\]
Suppose that \(w\)  satisfies $$w(\Gamma_R^S)=+\infty,$$ 
for any such \(S\subseteq S^{d-1}\) with $|S|>0$, and any $R>0$ (where $w(\Gamma_R^S)=\int_{\Gamma_R^S}w(x)dx$). Then, if a polynomial \(P\) vanishes at infinity with respect to \(w(x)dx\) the polynomial must be zero.
\end{lemma}
\begin{proof}
Let \(P\) be a non-zero polynomial. If \(P\) is a non-zero constant it is clear that it cannot vanish at infinity with respect to \(w\). So suppose \(P\) has degree \(k\in \N_1\). We can write
\[
P(x)=\sum_{|\alpha|\leq k}c_\alpha x^\alpha=\sum_{|\alpha|<k}c_\alpha x^\alpha +P_h(x),
\]
where \(P_h(x)\) is the homogeneous polynomial
\[
P_h(x)=\sum_{|\alpha|=k}c_\alpha x^\alpha. 
\]
and at least one of the coefficients $c_\alpha$ is not zero. Thus \(P_h\) is not zero. So there exists some point \(x_0\neq 0\) such that \(P_h(x_0)\neq 0\). By homogeneity we may assume that \(|x_0|=1\). By continuity we can find an open ball centered around \(x_0\) where \(P_h\) is never zero. Thus, it is possible to obtain a set \(S\subseteq S^{d-1}\) with \(|S|>0\) such that
\[
m:=\inf_{x\in S}|P_h(x)|>0.
\]
Now, for any \(x\neq 0\) such that \(x/|x|\in S\) we have the estimate
\[
\begin{split}
  |P(x)|&\geq |P_h(x)|-\sum_{|\alpha|<k}|c_\alpha||x|^{|\alpha|}\\
  &\geq m|x|^k-\sum_{|\alpha|<k}|c_\alpha||x|^{|\alpha|}.
\end{split}
\]
Given \(\lambda>0\), we can find \(R\) large enough so that 
\[
|x|\geq R\implies m|x|^k-\sum_{|\alpha|<k}|c_\alpha||x|^{|\alpha|}>\lambda.
\]
Therefore, 
\[
\Gamma_R^S\subseteq \{x: |P(x)|>\lambda\}. 
\]
This way it follows that
\[
w(\{x: |P(x)|>\lambda\})\geq w(\Gamma^S_R)=+\infty,
\]
which shows that \(P\) does not vanish at infinity with respect to \(w(x)dx\).
\end{proof}
\begin{remark}
    Note that, if \(w\) is a radial weight in \(A_\infty\), then it always satisfies the assumption of this lemma. Indeed, \(A_\infty\) weights satisfy an estimate of the form
    \[
    \frac{w(Q)}{w(A)} \gtrsim \left(\frac{|Q|}{|A|}\right)^\delta,
    \]
    for some \(\delta > 0\), where \(Q\) is any cube, and \(A\) is a positive measure subset of \(Q\). If we fix \(A\) and we let \(Q\) increase without bound we see that \(w(\R^d) = +\infty\). If \(w\) is radial, \(w(x) = w_0(|x|)\), this implies that
    \[
    \int_0^\infty w_0(\rho) \rho^{d-1}d\rho = +\infty.
    \]
    Moreover, since \(w\) is locally integrable, we know that also
    \[
    \int_R^\infty w_0(\rho)\rho^{d-1}d\rho = +\infty, \forall R>0.
    \]
    But then,
    \[
    w(\Gamma^S_R) = |S|\int_R^\infty w_0(\rho)\rho^{d-1}d\rho = +\infty. 
    \]
    Thus \(w\) satisfies the assumptions of the lemma. Since the weights we are mostly interested in are \(w(x) = |x|^\gamma\) and \(w(x) = \langle x\rangle^\gamma\), the technical condition of the lemma won't present any major difficulties.
\end{remark}
Using this lemma we may give a version of theorem \ref{thm2.2} with a stronger assumption that removes the polynomial.
\begin{cor}\label{cor2.1}
Let \(1<p\leq q<\infty\) and \(t<s<t+d\). Suppose that \((v,w)\in A_{p,q}^{s-t}\), \(v^{-p'/p}, w\in A_\infty\) and assume further that \(w(\Gamma^S_R)=+\infty\) for any \(S\subseteq S^{d-1}\) with \(|S|>0\) and for all \(R>0\). Then,
\[
\|D^tf\|_{L^q(w)}\lesssim \|D^sf\|_{L^p(v)},\ \forall f\in D^{>\tau}\mcal{S}(\R^d),
\]
where \(\tau=\max\{-d,-t\}\).
\end{cor}
\begin{proof}
Given \(f\in D^{>\tau}\mcal{S}(\R^d)\) we know by theorem \ref{thm2.2} that there exists some polynomial \(P\) such that
\[
\|D^tf-P\|_{L^q(w)}\lesssim \|D^sf\|_{L^p(v)}.
\]
So, if \(D^sf\in L^p(v)\), it follows that \(D^tf-P\in L^q(w)\). But \(D^tf\) goes to zero as \(|x|\rightarrow +\infty\) by the Riemann-Lebesgue lemma, and thus \(D^tf\) vanishes at infinity with respect to \(w(x)dx\). This implies that \(P\) vanishes at infinity with respect to \(w(x)dx\) and by lemma \ref{lemma2.4} this means that \(P\) must be zero.
\end{proof}

\subsection{The minimum exponent inequality}

Having considered the case \(\theta=1\) we now turn to the case \(\theta=t/s\). In other words, in this section we are interested in establishing the inequality
\begin{equation}\label{minexpineq}
\|D^tf\|_{L^r(w)}\lesssim \|D^sf\|_{L^p(u)}^{t/s}\|f\|_{L^q(v)}^{1-t/s}.
\end{equation}
We follow the strategy of \cite{Ponce}, lemma 5, and use ideas from complex interpolation, namely the three lines lemma. To be able to apply these ideas we need to first prove the estimate
\[
\|D^{i\tau}f\|_{L^p(w)}\lesssim_\tau \|f\|_{L^p(v)},\ \forall \tau\in \R. 
\]
Moreover, we need to be careful about the exact dependence of the implicit constant on the parameter \(\tau\). Some forms of interpolation, like Stein interpolation, allow for a doubly exponential growth rate of the constant as \(|\tau|\rightarrow +\infty\). The argument we give here works if the constant does not grow faster than exponentially and fortunately, it is possible to show that the dependence is polynomial, and so the simpler argument based on the three lines lemma will suffice. 

Note that the operator \(D^{i\tau}\) is a Fourier multiplier operator associated to a Hörmander multiplier. Much work has been done to extend the unweighted theory of Hörmander multipliers to the case of weighted inequalities. If one wishes to obtain only a one weight inequality then this follows from the work of Kurtz and Wheeden, see \cite{KW}. Obtaining a two-weight version is a more delicate matter, and a recent strategy that has been very fruitful is to use the method of domination by sparse operators. See for instance the paper \cite{Hytonen}, corollary 6.13, where a two weight inequality for Hörmander multipliers is obtained. The main tool used in that paper is Lerner's domination theorem (\cite{Lerner}). Here we will use the more recent version due to Lerner and Ombrosi, in theorem \ref{thm1.2}, so that one reduces the problem to that of showing that the auxiliary maximal operator is of weak type. In paper \cite{Hytonen} it is shown that the maximal operator is weak type \((r,r)\) for \(1<r\leq q\). This leads to assumptions on the weights that involve an \(A_{p,q}^r\) condition which is stronger than the \(A_{p,q}\) condition. Here we are dealing with a much more concrete case, and this will allow us to give a simple argument showing that \(M_\tau\) is weak type \((1, 1)\). In turn, this gives conditions which match more naturally the conditions from the previous section.  

The starting point is a well-known identity involving the Fourier transform of the kernel of \(D^{i\tau}\). We may write \(D^{i\tau}f=(|\cdot|^{i\tau}\hat{f})^\vee=(|\cdot|^{i\tau})^\vee * f= W_\tau*f\), where \(W_{\tau}=(|\cdot|^{i\tau})^\vee\in \mcal{S}'(\R^d)\). Now define
\[
\langle V_\tau,\phi\rangle:=\lim_{n}\int_{|x|>\delta_n}\frac{\phi(x)}{|x|^{d+i\tau}}dx.
\]
If \(\tau\neq 0\) and \(\delta_n=e^{-2\pi n/|\tau|}\), then this gives a well-defined tempered distribution. From \cite{Duo}, section 2 of chapter 5, or \cite{CGrafakos}, section 5.4.2, we have that
\begin{equation}\label{eq2.6}
W_\tau=\alpha_d(\tau)V_\tau + \beta_d(\tau)\delta,
\end{equation}
where
\[
\alpha_d(\tau)=\pi^{-\frac{d}{2}-i\tau}\frac{\Gamma\left(\frac{d+i\tau}{2}\right)}{\Gamma\left(-i\frac{\tau}{2}\right)}\text{ and }\beta_d(\tau)=i\frac{d|B_1|}{\tau}\alpha_d(\tau).
\]
As was mentioned above we must be careful about the order of growth of these constants as \(|\tau|\rightarrow +\infty\). To do this we use some standard facts about the Gamma function. If \(m\in \N_1\) and \(y\in \R\setminus \{0\}\), then
\begin{align*}
    |\Gamma(iy)|^2&=\frac{\pi}{y\sinh(\pi y)}\\
    |\Gamma(m+iy)|^2&=\frac{\pi y}{\sinh(\pi y)}\prod_{k=1}^{m-1}(k^2+y^2)\\
    |\Gamma\left(m-1/2+iy\right)|^2&=\frac{\pi}{\cosh(\pi y)}\prod_{k=1}^{m-1}[(k-1/2)^2+y^2].
\end{align*}
Now consider the case \(d=2m, m\in \N_1\). Using these identities we see that
\[
\begin{split}
    \frac{|\Gamma(m+i\tau/2)|}{|\Gamma(-i\tau/2)|}&=\frac{|\tau|}{2}\sqrt{1+\frac{\tau^2}{4}}\sqrt{2^2+\frac{\tau^2}{4}}\dots\sqrt{(m-1)^2+\frac{\tau^2}{4}}\\
    &\leq\frac{|\tau|}{2}\left(1+\frac{|\tau|}{2}\right)\dots \left(m-1+\frac{|\tau|}{2}\right),
\end{split}
\]
and the right hand side is a polynomial of \(|\tau|\) which has nonnegative coefficients and degree equal to \(m=\lceil d/2\rceil\). Now suppose \(d=2m-1, m\in \N_1\). Then we have that
\[
\begin{split}
    \frac{|\Gamma(m-1/2+i\tau/2)|}{|\Gamma(-i\tau/2)|}&=\sqrt{|\tanh(\pi\tau/2)||\tau|/2}\prod_{k=1}^{m-1}\sqrt{(k-1/2)^2+\frac{\tau^2}{4}}\\
    &\leq \pi^{\frac{1}{2}}\frac{|\tau|}{2}\prod_{k=1}^{m-1}\left[k-\frac{1}{2}+\frac{|\tau|}{2}\right].
\end{split}
\]
The right hand side is a polynomial of \(|\tau|\) which has nonnegative coefficients and degree equal to \(m=\lceil d/2\rceil\). This shows that there is a polynomial \(P_d\) with degree \(\lceil d/2\rceil\) and nonnegative coefficients such that
\[
|\alpha_d(\tau)|\leq P_d(|\tau|).
\]
Also, note that the polynomials we obtained are divisible by \(|\tau|\), which means we can also bound \(|\beta_d(\tau)|\) above by a polynomial, this time of degree \(\lceil d/2\rceil -1 \). 

Returning to identity \eqref{eq2.6}, and convolving with \(f\) we get
\begin{equation}\label{eq2.7}
D^{i\tau}f=\alpha_d(\tau)T_\tau f+\beta_d(\tau)f,\ \tau\in \R\setminus\{0\},\ f\in \mcal{S}(\R^d),
\end{equation}
where \(T_\tau\) is the operator defined by \(T_\tau f=V_\tau*f\). We may write
\[
T_\tau f(x)=\lim_{n}\int_{|y|>\delta_n}\frac{1}{|y|^{d+i\tau}}f(x-y)dy,
\]
so the kernel of this operator is \(K_\tau(x):=|x|^{-d-i\tau}\). It is not difficult to check that this kernel satisfies
\[
|K_\tau(x)|\leq \frac{1}{|x|^d}\text{ and }|K_\tau(x-y)-K_\tau(x)|\leq 2^{d+2}(d+|\tau|)\frac{|y|}{|x|^{d+1}},\text{ when }|x|\geq 2|y|.
\]
Therefore, the operator \(T_\tau\) is associated to a standard kernel. Moreover, since it is clear that \(D^{i\tau}\) is bounded on \(L^2\), then from identity \eqref{eq2.7} we conclude also that \(T_\tau\) is bounded on \(L^2\), since \(\alpha_d(\tau)\neq 0,\forall \tau\neq 0\). Thus, \(T_\tau\) is a Calderón-Zygmund operator. Then by standard Calderón-Zygmund theory,\footnote{See for instance theorem 5.3.3 from \cite{CGrafakos} and section 4 of Chapter 5 of \cite{Duo}.} \(T_\tau\) has an extension to \(L^1\) that maps it to \(L^{1,\infty}\) such that
\[
\|T_\tau\|_{L^1\rightarrow L^{1,\infty}}\lesssim_d1+|\tau|+\frac{1+|\beta_d(\tau)|}{|\alpha_d(\tau)|},
\]
and
\[
T_\tau f(x)=\lim_n\int_{|y|\geq \delta_n}K_\tau(y)f(x-y)dy\text{ for a.e. }x\in \R^d,\forall f\in L^1(\R^d).
\]
In particular, this means that \(D^{i\tau}\) also extends to \(L^1\) and 
\[
\|D^{i\tau}\|_{L^1\rightarrow L^{1,\infty}}\lesssim_d|\alpha_d(\tau)|+|\tau||\alpha_d(\tau)|+1+2|\beta_d(\tau)|.
\]
We have seen above that both \(|\alpha_d(\tau)|\) and \(|\beta_d(\tau)|\) can be controlled by polynomials, and therefore we see that there is a polynomial \(P_d^{1}\) with nonnegative coefficients depending only on \(d\), and with degree \(\lceil d/2\rceil +1\), such that
\begin{equation}\label{eq2.8}
\|D^{i\tau}\|_{L^1\rightarrow L^{1,\infty}}\leq P_d^{1}(|\tau|),\ \forall \tau\in \R.
\end{equation}
We proved this for \(\tau\neq 0\), but of course it still works for \(\tau=0\) for in that case the left hand side is bounded by one. Now we define the maximal operator associated with \(D^{i\tau}\) as 
\[
M_\tau f(x)=\sup_{Q\ni x}\sup_{x',x''\in Q}|D^{i\tau}(f\chi_{\R^d\setminus Q^*})(x')-D^{i\tau}(f\chi_{\R^d\setminus Q^*})(x'')|,
\]
while the maximal operator associated with \(T_\tau\) will be denoted by \(\mathcal{M}_\tau \), i.e.
\[
\mathcal{M}_\tau f(x)=\sup_{Q\ni x}\sup_{x',x''\in Q}|T_\tau(f\chi_{\R^d\setminus Q^*})(x')-T_\tau(f\chi_{\R^d\setminus Q^*})(x'')|.
\]
From \eqref{eq2.7} we know that \(M_\tau f(x)=|\alpha_d(\tau)|\mathcal{M}_\tau f(x)\), so it suffices to estimate \(\mathcal{M}_\tau\).
\begin{lemma}\label{lemma2.6}
We have the estimate
\[
\|\mathcal{M}_\tau \|_{L^1\rightarrow L^{1,\infty}}\lesssim_d 1+|\tau|,\ \tau\neq 0.
\]
\end{lemma}
\begin{proof}
This follows from the fact that \(T_\tau\) is a Calderón-Zygmund operator. See for instance \cite{Lerner}. For completeness we include the short proof. Let \(f\in L^1(\R^d)\) and \(x\in \R^d\). Fix a cube \(Q\ni x\) with side length \(l\) and take \(x',x''\in Q\). We have that
\[
\begin{split}
  |T_\tau(f\chi_{\R^d\setminus Q^*})(x')&-T_\tau(f\chi_{\R^d\setminus Q^*})(x'')|\leq \int_{\R^d\setminus Q^*}|K_\tau(x'-y)-K_\tau(x''-y)||f(y)|dy\\
  &\leq \sum_{k\geq 0}\int_{2^{k+1}Q^*\setminus 2^kQ^*}2^{d+2}(d+|\tau|)\frac{|x'-x''|}{|x''-y|^{d+1}}|f(y)|dy\\
  &\lesssim_d(1+|\tau|)\text{diam}(Q)\sum_{k\geq 0}\frac{(5\sqrt{d}2^{k+1}l)^d}{l^{d+1}2^{k(d+1)}}\frac{1}{|2^{k+1}Q^*|}\int_{2^{k+1}Q^*}|f(y)|dy\\
  &\lesssim_d(1+|\tau|)\sum_{k\geq 0}\frac{1}{2^k}Mf(x)\lesssim_d(1+|\tau|)Mf(x).
\end{split}
\]
This shows that 
\[
\mathcal{M}_\tau f(x)\lesssim_d(1+|\tau|)Mf(x)\text{ for a.e. }x.
\]
Therefore,
\[
\|\mathcal{M}_\tau f\|_{L^{1,\infty}}\lesssim_d(1+|\tau|)\|Mf\|_{L^{1,\infty}}\lesssim_d(1+|\tau|)\|f\|_{L^1}.
\]
\end{proof}
From this lemma we see that there exists a polynomial \(P_d^{2}\) with nonnegative coefficients that depend only on \(d\), and with degree \(\lceil d/2\rceil+1\) such that 
\begin{equation}\label{eq2.9}
\|M_\tau\|_{L^1\rightarrow L^{1,\infty}}\leq P_d^{2}(|\tau|),\ \forall \tau\in \R. 
\end{equation}
Again, we proved this for the case \(\tau\neq 0\), but for \(\tau=0\) the left hand side is zero so that case is trivial. The next result now follows immediately from \eqref{eq2.8}, \eqref{eq2.9} and Lerner's Domination theorem \ref{thm1.2}.
\begin{prop}\label{prop2.4}
For every compactly supported \(f\in L^1(\R^d)\), there exists a \((2\cdot 5^d\sqrt{d}^d)^{-1}\)-sparse family \(\mcal{S}=\mcal{S}(f)\) such that 
\[
|D^{i\tau}f(x)|\leq P_d(|\tau|)\mcal{A}_\mcal{S}f(x)\text{ for a.e. }x\in \R^d,\footnote{We use the notation \(\mcal{A}_\mcal{S}\) for \(\mcal{A}_{1,\mcal{S}}\).}
\]
where \(P_d\) is a polynomial with nonnegative coefficients that depend only on \(d\), and with degree equal to \(\lceil d/2\rceil +1\).
\end{prop}

We now wish to apply theorem \ref{bound for sparse operators}. To do this we note that, given any sparse family of cubes \(\S\), there are dyadic lattices \(\mcal{D}^1,\dots, \mcal{D}^{3^d}\) and sparse families \(\mcal{S}_j=\mcal{S}_j(f)\subseteq \mcal{D}^j\) such that
\[
\mcal{A}_\mcal{S}^\alpha f(x)\lesssim_{\alpha, d} \sum_{j=1}^{3^d}\mcal{A}_{\mcal{S}_j}^\alpha f(x).
\]
This is remark 4.3 in \cite{Lerner}. We now finally obtain the desired result.
\begin{thm}\label{thm2.3}
If \(1<p<\infty\), \((v,w)\in A_{p,p}\) and \(v^{-p'/p},w\in A_\infty\), then
\[
\|D^{i\tau}f\|_{L^p(w)}\lesssim_{p,v,w} P_d(|\tau|)\|f\|_{L^p(v)},\ \forall f\in D^{\geq 0}\mcal{S}(\R^d),\ \forall \tau\in \R,
\]
where \(P_d\) is a polynomial with nonnegative coefficients that depend only on \(d\), and with degree equal to \(\lceil d/2\rceil +1\).
\end{thm}
\begin{proof}
Let \(f\in D^{\geq 0}\mcal{S}(\R^d)\) and put \(f_N=f\chi_{B(0,N)}\). From what we've seen above there are sparse families \(\mcal{S}_j=\mcal{S}_j(f_N)\) coming from dyadic lattices, such that 
\[
|D^{i\tau}f_N(x)|\leq P_d(|\tau|)\sum_{j=1}^{3^d}\mcal{A}_{\mcal{S}_j}f_N(x)\text{ for a.e. }x\in \R^d.
\]
By theorem \ref{bound for sparse operators} it follows that 
\[
\|D^{i\tau}f_N\|_{L^p(w)}\leq P_d(|\tau|)\sum_{j=1}^{3^d}\|\mcal{A}_{\mcal{S}_j}f_N\|_{L^p(w)}\lesssim P_d(|\tau|)\sum_{j=1}^{3^d}\|f_N\|_{L^p(v)}\lesssim P_d(|\tau|)\|f\|_{L^p(v)}.
\]
Now, since \(f_N\rightarrow_Nf\) in \(L^2\), then \(D^{i\tau}f_N\rightarrow_ND^{i\tau}f\) in \(L^2\). Therefore we can pick a subsequence \((N_k)_k\) such that \(D^{i\tau}f_{N_k}(x)\rightarrow_kD^{i\tau}f(x)\) for a.e. \(x\). Then, by Fatou's lemma,
\[
\|D^{i\tau}f\|_{L^p(w)}\leq \liminf_k\|D^{i\tau}f_{N_k}\|_{L^p(w)}\lesssim_{p,v,w}P_d(|\tau|)\|f\|_{L^p(v)}.
\]
\end{proof}
Now that we have established theorem \ref{thm2.3} we can prove the minimum exponent inequality \eqref{minexpineq}.
\begin{thm}\label{thm2.4}
Let \(1<p,q,r<\infty\) and \(0<t<s\) be such that
\[
\frac{1}{r}=\frac{t}{s}\frac{1}{p}+\left(1-\frac{t}{s}\right)\frac{1}{q}.
\]
Suppose we have a family of weights, \(\{w_z\}_{z\in \overline{S}}\), where \(S = \{z\in \C: 0< \Re(z)<1\}\), that satisfies the following properties:
\begin{enumerate}
    \item There is some \(h\in L^1_\text{loc}(\R^d)\) such that \(|w_z(x)|\leq h(x), \forall x,\forall z\);
    \item For each \(x\in \R^d\) the function \(z\mapsto w_z(x)\) is continuous in \(\overline{S}\) and analytic in \(S\);
    \item There are weights \(w_0, w_1\), such that \(|w_{i\tau}(x)|\leq w_0(x)\) and \(|w_{1+i\tau}(x)|\leq w_1(x)\) uniformly on \(\tau\).
\end{enumerate}
Moreover, assume that \(u, v, w\) are weights such that \((u^p, w_1^p)\in A_{p, p}, (v^q, w_0^q)\in A_{q, q}, w = w_{t/s},\) and \(u^{-p'}, w_1^p, v^{-q'}, w_0^q\in A_\infty\). Then, 
\[
\|wD^tf\|_{L^r}\lesssim \|uD^sf\|_{L^p}^{t/s}\|vf\|_{L^q}^{1-t/s},\ \forall f\in \mcal{S}(\R^d).
\]
\end{thm}
\begin{proof}
Let \(g\in C^\infty_c(\R^d)\) with \(\|g\|_{L^{r'}}=1\), let \(f\in \mcal{S}(\R^d)\) and put 
\[
F(z)=e^{z^2-1}\int_{\R^d}D^{zs}f|g|^{P(z)}\text{sgn}(g)w_z(x)dx,
\]
where
\[
P(z)=\frac{r'}{p'}z+\frac{r'}{q'}(1-z).
\]
We need to show that \(F\) is bounded and continuous in \(\overline{S}\), and analytic in \(S\). We begin by showing that \(F\) is bounded. First note that 
\[
\|D^{zs}f\|_{L^\infty}\leq \int |\xi|^{\mbox{Re}(zs)}|\hat{f}(\xi)|d\xi\leq \|\hat{f}\|_{L^1}+\||\xi|^s\hat{f}\|_{L^1}=:A<\infty,
\]
with \(A\) independent of \(z\). So, we have that
\[
|F(z)|\leq A\int_{\supp(g)}|g|^{\frac{r'}{p'}\mbox{Re}(z)+\frac{r'}{q'}(1-\mbox{Re}(z))}|w_z(x)|dx\leq A\int_{\supp(g)}(1+|g|^{\frac{r'}{p'}+\frac{r'}{q'}})h(x)dx<\infty .
\]
This shows that \(F\) is bounded on \(\overline{S}\). Now recall that we may write
\[
D^{zs}f(x)=\int_{\xi\neq 0}e^{2\pi i x\cdot \xi}|\xi|^{zs}\hat{f}(\xi)d\xi 
\]
Since \(|e^{2\pi i x\cdot \xi}|\xi|^{zs}\hat{f}(\xi)|\leq (1+|\xi|^s)|\hat{f}(\xi)|\in L^1\) and \(|\xi|^{zs}\) is continuous in \(z\), it follows by the dominated convergence theorem that \(D^{zs}f(x)\) is continuous in \(z\) on \(\overline{S}\) for each fixed \(x\). Moreover, the estimate \(|e^{2\pi i x\cdot \xi}|\xi|^{zs}\hat{f}(\xi)|\leq (1+|\xi|^s)|\hat{f}(\xi)|\in L^1\) shows that we may apply Fubini to obtain
\[
\oint_\gamma D^{zs}f(x)dz=\oint_\gamma \int_{\xi\neq 0}e^{2\pi i x\cdot \xi}|\xi|^{zs}\hat{f}(\xi)d\xi dz=\int_{\xi\neq 0} e^{2\pi i x\cdot \xi}\hat{f}(\xi)\oint_\gamma |\xi|^{zs}dzd\xi=0,
\]
where \(\gamma\) is any closed \(C^1\) curve in \(S\). By Morera's theorem it follows that \(D^{zs}f(x)\) is analytic in \(z\) on \(S\). Now we argue similarly for \(F\). First note that from the estimate
\[
|D^{zs}f|g|^{P(z)}\text{sgn}(g)w_z(x)|\leq A(1+|g|^{\frac{r'}{p'}+\frac{r'}{q'}})h\chi_{\supp(g)}\in L^1
\]
and the fact that \(D^{zs}f(x)|g(x)|^{P(z)}w_z(x)\) is continuous in \(z\) it follows that \(F(z)\) is continuous on \(\overline{S}\) by the dominated convergence theorem. Moreover the same estimate shows that we can apply Fubini's theorem to obtain
\[
\oint_\gamma F(z)dx=\int \text{sgn}(g)\oint_\gamma e^{z^2-1}D^{zs}f|g|^{P(z)}w_z(x)dzdx=0,
\]
so again by Morera's theorem we may conclude that \(F\) is analytic in \(S\). Now, using theorem \ref{thm2.3} we have the two estimates 
\begin{align*}
    |F(i\tau)|&\leq e^{-\tau^2}\|w_0D^{i\tau s}f\|_{L^q}\||g|^{r'/q'}\|_{L^{q'}}\lesssim_{q,v,w}e^{-\tau^2}P_d(|\tau s|)\|vf\|_{L^q}\lesssim_{q,d,v,w,s}\|vf\|_{L^q};\\
    |F(1+i\tau)|&\leq e^{-\tau^2}\|w_1D^{i\tau s}D^sf\|_{L^p}\||g|^{r'/p'}\|_{L^{p'}}\lesssim_{p,u,w}e^{-\tau^2}P_d(|\tau s|)\|uD^sf\|_{L^p}\\
    &\lesssim_{p,d,u,w,s}\|uD^sf\|_{L^p}.
\end{align*}
Therefore, by the three lines lemma,
\[
|F(z)|\lesssim_{p,q,d,u,v,w,s,\theta}\|uD^sf\|_{L^p}^{\theta}\|vf\|_{L^q}^{1-\theta}\text{ for all }z\text{ with }\mbox{Re}(z)=\theta\in ]0,1[.
\]
Setting \(z=t/s\) we get
\[
\left|\int wD^tfgdx\right|\lesssim_{p,q,d,u,v,w,s,t}\|uD^sf\|_{L^p}^{t/s}\|vf\|_{L^q}^{1-t/s}. 
\]
Taking the supremum over all \(g\in C^\infty_c\) with \(\|g\|_{L^{r'}}=1\), the result follows. 
\end{proof}

In the applications we look at in section \ref{particular}, the following corollary will prove to be enough.

\begin{cor}\label{cor2.15}
    Let \(1<p,q,r<\infty\), \(0<t<s\) and \(\alpha, \beta, \gamma\in \R\) be such that
$$
\frac{1}{r}=\frac{t}{s}\frac{1}{p}+\left(1-\frac{t}{s}\right)\frac{1}{q}$$
and
$$\gamma=\frac{t}{s}\alpha+\left(1-\frac{t}{s}\right)\beta.$$
Suppose also that \(w\) is such that \(w^\alpha, w^\beta\) are weights and \(w^{\alpha p}\in A_p\), \(w^{\beta q}\in A_q\). Then,
\[
\|w^\gamma D^tf\|_{L^r}\lesssim \|w^\alpha D^sf\|_{L^p}^{t/s}\|w^\beta f\|_{L^q}^{1-t/s},\ \forall f\in \mcal{S}(\R^d).
\]
\end{cor}
\begin{proof}
    Apply theorem \ref{thm2.4} with \(w_z(x) = w(x)^{z\alpha + (1-z)\beta}\), \(h(x) = w(x)^\alpha + w(x)^\beta, w_0(x) = v(x) = w(x)^\beta, w_1(x) = u(x) = w(x)^\alpha\).
\end{proof}

\subsection{Intermediate exponents}

In this section we use theorems \ref{thm2.2} and \ref{thm2.4} to prove the Gagliardo-Nirenberg inequality for
the intermediate range of values of the interpolation parameter, between those two end points.
\begin{thm}\label{thm2.5}
Let \(1<p,q,r<\infty, 0<t<s\), \(t/s < \theta < 1\) and define \(a\in]1,+\infty[\) by 
\[
\frac{1}{a}=\frac{\theta}{p}+\frac{1-\theta}{q}.
\]
Suppose that \(a\leq r\) and \(\theta s < t + d\). Assume there is some weight \(\tilde{w}\) such that \((\tilde{w}^a, w^r)\in A_{a,r}^{\theta s-t}\) and \(\tilde{w}^{-a'}, w^r\in A_\infty\). Moreover, assume that there is a family \(\{w_z\}_{z\in \overline{S}}\) satisfying the assumptions of theorem \ref{thm2.4} with \(\tilde{w}=w_\theta\), \((u^p, w_1^p)\in A_{p,p}, (v^q, w_0^q)\in A_{q, q}\) and \(u^{-p'}, w_1^p, v^{-q'}, w_0^q\in A_\infty\). Then, for each \(f\in \mcal{S}(\R^d)\) there is some polynomial \(P\) with degree at most \(\lfloor \theta s - t\rfloor\) such that 
\[
\|w(D^tf-P)\|_{L^r}\lesssim \|uD^sf\|_{L^p}^\theta \|vf\|_{L^q}^{1-\theta}.
\]
Furthermore, if \(w^r(\Gamma^S_R)=+\infty\) for all $S\subseteq S^{d-1}$, such that $|S|>0$, and all $R>0$, then \(P=0\).
\end{thm}
\begin{proof}
Note that \(1<a\leq r<\infty, 0<t<\theta s\), \((\tilde{w}^a,w^r)\in A_{a,r}^{\theta s -t}\) and \(\tilde{w}^{-a'}, w^r\in A_\infty\). From theorem \ref{thm2.2} we conclude the existence of some polynomial \(P\) with degree at most \(\lfloor \theta s-t\rfloor\) such that
\[
\|D^tf-P\|_{L^r(w^r)}\lesssim \|D^{\theta s}f\|_{L^a(\tilde{w}^a)}.
\]
Note also that \(0<\theta s<s\), \(1/a=\theta/p+(1-\theta)/q\), \((u^p, w_1^p)\in A_{p,p}, (v^q, w_0^q)\in A_{q, q}\) and \(u^{-p'}, w_1^p, v^{-q'}, w_0^q\in A_\infty\). Therefore, by theorem \ref{thm2.4} we see that
\[
\|\tilde{w}D^{\theta s}f\|_{L^a}\lesssim \|uD^sf\|_{L^p}^\theta \|vf\|_{L^q}^{1-\theta}.
\]
The result now follows.
\end{proof}

\section{Particular Families of Weights}\label{particular}

In this section we apply the theorems from the previous section to a couple of particular cases, namely power-law weights \(w(x) = |x|^\gamma\) and japanese bracket weights \(w(x) = \langle x\rangle^\gamma\)

\subsection{Homogeneous weights}\label{Homogeneous weights}

In this section we devote our attention to weights of the form \(w(x)=|x|^\gamma,\ \gamma>-d\). Our goal is to apply theorems \ref{thm2.2}, \ref{thm2.4} and \ref{thm2.5} to homogeneous weights. For that, we must first study when these weights satisfy the \(A_{p,q}^\alpha\) condition and the \(A_\infty\) condition. 

Let us begin by looking at the \(A_{p,q}^\alpha\) condition. Let \(1<p,q<\infty\) and \(\alpha\in\R\). Our goal is to find all triples \((a,b,\alpha)\) such that 
\[
\sup_Q |Q|^{\frac{\alpha}{d}-1}\left(\int_Q|x|^a\right)^{1/p'}\left(\int_Q|x|^b\right)^{1/q}<\infty.
\]
First and foremost the functions need to be locally integrable so we need \(a>-d\) and \(b>-d\). Clearly the triples that satisfy this condition are the same as the triples that satisfy the corresponding condition for balls. This way, we consider balls instead of cubes. Here we use an idea as in \cite{CGrafakos}, which is to split between balls of type I and balls of type II. A ball \(B=B(x_0,r)\) is said to be of type I if \(2r<|x_0|\) and of type II if \(2r\geq |x_0|\). For balls of type I the radius is small so we can compare all points to the center point, whereas balls of type II can be compared to balls centered at the origin. Suppose \(B=B(x_0,r)\) is a ball of type I. Then, if \(x\in B\), we have that
\[
\frac{1}{2}|x_0|\leq |x-x_0|+|x|-\frac{1}{2}|x_0|\leq |x|\leq |x-x_0|+|x_0|\leq \frac{3}{2}|x_0|.
\]
so \(|x|\sim |x_0|,\forall x\in B\). This means that
\[
|B|^{\frac{\alpha}{d}-1}\left(\int_B|x|^a\right)^{1/p'}\left(\int_B|x|^b\right)^{1/q}\sim |B|^{\frac{\alpha}{d}-1+\frac{1}{p'}+\frac{1}{q}}|x_0|^{\frac{a}{p'}+\frac{b}{q}}\sim r^{\alpha-\frac{d}{p}+\frac{d}{q}}|x_0|^{\frac{a}{p'}+\frac{b}{q}}.
\]
If we fix \(x_0\neq 0\) and let \(r\rightarrow 0^+\) we see that the condition 
\[
\alpha-\frac{d}{p}+\frac{d}{q}\geq0
\]
is necessary. Also, if we pick \(r=|x_0|/4, x_0\neq 0\), we see that for this to be bounded we need the condition 
\begin{equation}\label{eq3.2}
\alpha-\frac{d}{p}+\frac{d}{q}+\frac{a}{p'}+\frac{b}{q}=0.
\end{equation}
If both of these conditions are satisfied then the supremum over type I balls is finite. If instead we consider a type II ball \(B=B(x_0,r)\), then since \(B\subseteq B(0,3r)\) this leads us to consider balls that are centered at the origin. We have that
\[
\begin{split}
|B(0,R)|^{\frac{\alpha}{d}-1}\left(\int_{B(0,R)}|x|^a\right)^\frac{1}{p'}\left(\int_{B(0,R)}|x|^b\right)^\frac{1}{q}&\sim R^{\alpha-d}\left(\int_0^R\rho^{a+d-1}d\rho\right)^\frac{1}{p'}\left(\int_0^R\rho^{b+d-1}d\rho\right)^\frac{1}{q}\\
&\sim R^{\alpha-d+\frac{a+d}{p'}+\frac{b+d}{q}}.
\end{split}
\]
So, this quantity remains bounded \(\forall R>0\) if and only if condition \eqref{eq3.2} is satisfied. These considerations lead us to the following lemma.
\begin{lemma}\label{lemma3.1}
Suppose \(1<p,q<\infty\) and \(\alpha\in \R\). Then,
\[
(|x|^a, |x|^b)\in A_{p,q}^\alpha \iff \alpha\geq \frac{d}{p}-\frac{d}{q},\  \alpha-\frac{d}{p}+\frac{d}{q}=\frac{a}{p}-\frac{b}{q},\ a<d(p-1)\text{ and }b>-d.
\]
\end{lemma}
Noting that the \(A_{p,q}\) condition is just the \(A_{p,q}^\alpha\) condition with \(\alpha=d/p-d/q\), it follows that
\[
(|x|^a, |x|^b)\in A_{p,q} \iff \frac{a}{p}=\frac{b}{q},\ a<d(p-1)\text{ and }b>-d.
\]
In particular,
\[
|x|^a\in A_p\iff (|x|^a,|x|^a)\in A_{p,p}\iff -d<a<d(p-1).
\]
This shows that homogeneous weights always satisfy some \(A_p\) condition, and therefore they always satisfy the \(A_\infty\) condition. With this we can now apply corollary \ref{cor2.1}, corollary \ref{cor2.15} and the argument in the proof of theorem \ref{thm2.5} with \(\tilde{w} = |x|^\mu, \mu = \theta \alpha + (1-\theta)\beta\), to yield the following theorem.
\begin{thm}\label{thm3.1}
Suppose that \(1<p,q,r<\infty, 0\leq t<s,\, t/s\leq\theta\leq 1, \alpha\in]-d/p,d/p'[,\beta\in]-d/q,d/q'[,\gamma>-d/r\). Assume these satisfy
$$
    \frac{1}{r}-\frac{t-\gamma}{d}=\theta\left(\frac{1}{p}-\frac{s-\alpha}{d}\right)+(1-\theta)\left(\frac{1}{q}+\frac{\beta}{d}\right)$$
    and
    $$0\leq \theta \alpha+(1-\theta)\beta-\gamma\leq \theta s-t.$$
Then it follows that
\[
\||x|^\gamma D^tf\|_{L^r(\R^d)}\lesssim \||x|^\alpha D^sf\|_{L^p(\R^d)}^\theta\||x|^\beta f\|_{L^q(\R^d)}^{1-\theta},\ \forall f\in \mcal{S}(\R^d).
\]
\end{thm}
\begin{remark}
We can include the case \(t=0\) here because \(\S(\R^d)\subseteq L^{r'}(|x|^{-\gamma r'})\) (see remark \ref{remark1}). This is because, according to remark \ref{remark on the size of gamma}, \(\gamma < d/r'\). 
\end{remark}
\begin{remark}
Note that this theorem generalizes both the Stein-Weiss inequality (\cite{SteinWeiss}) and Lin's inequality (\cite{Lin}) (at least when \(r>1\)). 
\end{remark}
\begin{remark}
In the original work of Lin, the condition \(\theta \alpha + (1-\theta)\beta - \gamma \leq \theta s - t\) is only present if 
\[
\frac{1}{p}-\frac{s-\alpha}{d} = \frac{1}{q}+\frac{\beta}{d}.
\]
Whether this is not true in the setting of Riesz derivatives or whether our argument can be improved to recover this range of parameters remains unclear. In our case, this condition comes from the requirement that \(p\leq q\) in theorem \ref{thm2.2}.
\end{remark}

\subsection{Non-homogeneous weights}

In this section we apply the theorems of section \ref{general} to the weights \(w(x)=\langle x\rangle^\gamma\). Again, to do this we simply have to check when such weights satisfy the \(A_{p,q}^\alpha\) condition and the \(A_\infty\) condition. By splitting between type I and type II balls it is not hard to see that
\[
\langle x\rangle^\gamma\in A_p\iff -d<\gamma<d(p-1).
\]
This means that the weights \(\langle x\rangle^\gamma\) are in \(A_\infty\) if and only if \(\gamma>-d\). For this reason we restrict our attention to this range of exponents.

Let us now check when the \(A_{p,q}^\alpha\) condition holds. Let \(1<p,q<\infty\), \(\alpha\in \R\) and \(a,b>-d\). We are interested in the triples \((a,b,\alpha)\) such that
\[
\sup_{B}|B|^{\frac{\alpha}{d}-1}\left(\int_B\langle x\rangle^a\right)^{1/p'}\left(\int_B\langle x\rangle^b\right)^{1/q}<\infty. 
\]
Suppose \(B=B(x_0,r)\) is a type I ball, i.e. \(2r<|x_0|\). Then,
\[
|B|^{\frac{\alpha}{d}-1}\left(\int_B\langle x\rangle^a\right)^{1/p'}\left(\int_B\langle x\rangle^b\right)^{1/q}\sim r^{\alpha-\frac{d}{p}+\frac{d}{q}}\langle x_0\rangle^{\frac{a}{p'}+\frac{b}{q}}.
\]
By fixing \(x_0\) and letting \(r\rightarrow 0\) we see that the condition \(\alpha-d/p+d/q\geq 0\) is necessary. Also, choosing \(r=\langle x_0\rangle /4, |x_0|>1\), and letting \(|x_0|\rightarrow +\infty\) we see that the condition 
\[
\alpha-\frac{d}{p}+\frac{d}{q}+\frac{a}{p'}+\frac{b}{q}\leq 0
\]
is also necessary. These two conditions combined are sufficient to control all type I balls and since they are necessary we assume they hold from now on. Since type II balls can be compared to balls centered at the origin we now consider a ball \(B=B(0,r)\). Put
\[
\Phi(r,a)=\int_{B(0,r)}\langle x\rangle^adx=d|B_1|\int_0^r\rho^{d-1}(1+\rho^2)^{a/2}d\rho.
\]
It is clear that
\[
\lim_{r\rightarrow 0}\frac{\Phi(r,a)}{r^d}=|B_1|.
\]
Therefore, 
\[
|B(0,r)|^{\frac{\alpha}{d}-1}\Phi(r,a)^{1/p'}\Phi(r,b)^{1/q}\sim r^{\alpha-\frac{d}{p}+\frac{d}{q}}\left(\frac{\Phi(r,a)}{r^d}\right)^{1/p'}\left(\frac{\Phi(r,b)}{r^d}\right)^{1/q}
\]
has finite limit as \(r\rightarrow 0\) if and only if \(\alpha-d/p+d/q\geq 0\). Now, since \(a>-d\), then for \(r>1\) we have that
\[
\begin{split}
\Phi(r,a)&=d|B_1|\int_0^1\rho^{d-1}(1+\rho^2)^{a/2}d\rho+d|B_1|\int_1^r\rho^{d-1}(1+\rho^2)^{a/2}d\rho \\
&\sim_ad|B_1|\int_0^1\rho^{d-1}d\rho+d|B_1|\int_1^r\rho^{d-1+a}d\rho\\
&\sim_a|B_1|+d|B_1|\frac{r^{a+d}-1}{a+d}\sim_{a,d}\frac{a}{a+d}+\frac{d}{a+d}r^{a+d}.
\end{split}
\]
Using this we get for \(r>1\),
\[
\begin{split}
r^{\alpha-d}\Phi(r,a)^{1/p'}\Phi(r,b)^{1/q}&\sim r^{\alpha-d}\left( \frac{a}{a+d}+\frac{d}{a+d}r^{a+d}\right)^{1/p'}\left(\frac{b}{b+d}+\frac{d}{b+d}r^{b+d}\right)^{1/q}\\
&\sim \left[ \frac{a}{a+d}\left( \frac{b}{b+d}r^{q(\alpha-d)}+\frac{d}{b+d}r^{q(\alpha-d)+b+d} \right)^{\frac{p'}{q}}+\right.\\
&\left.\frac{d}{a+d}\left( \frac{b}{b+d}r^{q(\alpha-d)+\frac{q}{p'}(a+d)}+\frac{d}{b+d}r^{q(\alpha-d)+b+d+\frac{q}{p'}(a+d)} \right)^{\frac{p'}{q}}  \right]^{\frac{1}{p'}}
\end{split}
\]
which remains bounded as \(r\rightarrow +\infty\) if and only if \(\alpha\leq d,\ q(\alpha-d)+b+d\leq 0\) and \(q(\alpha-d)+q(a+d)/p'\leq 0\). Therefore we have shown the following lemma.

\begin{lemma}\label{lemma3.2}
Suppose \(1<p,q<\infty\), \(\alpha\in \R\) and \(a<d(p-1),\ b>-d\). Then, the pair \((\langle x\rangle^a, \langle x\rangle^b)\) satisfies the \(A_{p,q}^\alpha\) condition if and only if
\[
\alpha\geq \frac{d}{p}-\frac{d}{q},\ \alpha-\frac{d}{p}+\frac{d}{q}\leq \frac{a}{p}-\frac{b}{q},\ \alpha\leq d,\ \frac{b}{q}\leq \frac{d}{q'}-\alpha,\ \frac{a}{p}\geq \alpha-\frac{d}{p}.
\]
\end{lemma}
We now apply again corollary \ref{cor2.1}, corollary \ref{cor2.15} and theorem \ref{thm2.5} with \(\tilde{w}(x) = \langle x\rangle^{\mu}, \mu = \alpha \theta + \beta(1-\theta)\). This yields the following result:
\begin{thm}\label{thm3.2}
Suppose that \(1<p,q,r<\infty,\, 0\leq t<s,\,\, t/s \leq \theta \leq 1, \alpha\in]-d/p,d/p'[,\beta\in]-d/q,d/q'[,\gamma>-d/r\). Assume these satisfy
\begin{align*}
    \frac{1}{r}&\leq \frac{\theta}{p} + \frac{1-\theta}{q}, \\
    \gamma &\leq \theta \alpha+(1-\theta)\beta
\end{align*}
and
\[
\theta\left(\frac{1}{p}-\frac{s}{d}\right)+\frac{1-\theta}{q}\leq \frac{1}{r}-\frac{t}{d}\leq \theta\left(\frac{1}{p}-\frac{s-\alpha}{d}\right)+(1-\theta)\left(\frac{1}{q}+\frac{\beta}{d}\right)-\frac{\gamma}{d}.
\]
Then it follows that
\[
\|\langle x\rangle^\gamma D^tf\|_{L^r(\R^d)}\lesssim\|\langle x\rangle^\alpha D^sf\|_{L^p(\R^d)}^\theta\|\langle x\rangle^\beta f\|_{L^q(\R^d)}^{1-\theta},\ \forall f\in \mcal{S}(\R^d).
\]
\end{thm}

\section{Non-homogeneous derivatives}\label{nonhomderivatives}

In this section our goal is to obtain the same Gagliardo-Nirenberg inequalities we have obtained so far, but we are interested in non-homogeneous Bessel type derivatives now. That is, we define the operator \(J^z\) by
\[
J^zf=(\langle \cdot\rangle^z\hat{f})^\vee,
\]
which is a well defined continuous linear operator on tempered distributions \(f\in \mcal{S}'(\R^d)\). As before, the crucial estimates are
\begin{align}
\|J^{-\alpha}f\|_{L^q(w)}&\lesssim \|f\|_{L^p(v)} \label{eq4.1}\\
\|J^{i\tau}f\|_{L^p(w)}&\lesssim C(\tau)\|f\|_{L^p(v)}.\label{eq4.2}
\end{align}
We start by considering estimate \eqref{eq4.1}. The idea is to compare \(J^{-\alpha}\) to \(\I_\alpha\). We may write \(J^{-\alpha}f= G_\alpha *f\), where \(G_\alpha=(\langle \cdot\rangle^{-\alpha})^\vee\). Several properties of \(G_\alpha\) are known and we collect them in the following proposition. For a reference see for instance \cite{MGrafakos}. 
\begin{prop}\label{prop4.1}
For \(\alpha>0\), \(G_\alpha\in L^1, \|G_\alpha\|_{L^1}=1, G_\alpha>0, G_\alpha\in C^\infty(\R^d\setminus\{0\})\) and 
\begin{align*}
    G_\alpha(x)&\lesssim_{\alpha,d}e^{-\pi |x|},\ |x|\geq 1/\pi;\\
    G_\alpha(x)&\lesssim_{\alpha,d}\begin{dcases}
    |x|^{\alpha-d},&0<\alpha<d\\
    \log(1/\pi|x|)+1,&\alpha=d\\
    1,&\alpha>d
    \end{dcases},\ |x|\leq 1/\pi
\end{align*}
\end{prop}
This way, if \(0<\alpha<d\) it follows that \(G_\alpha(x)\lesssim_{\alpha,d}|x|^{\alpha-d}\). So, if \(f\in \mcal{S}(\R^d)\),
\[
|J^{-\alpha}f(x)|\leq \int G_{\alpha}(y)|f(x-y)|dy\lesssim \int |y|^{\alpha-d}|f(x-y)|dy\lesssim \I_\alpha( |f|)(x).
\]
This comparison leads immediately to the following result.
\begin{thm}\label{thm4.1}
Let \(1<p\leq q<\infty\) and \(0<\alpha <d\). Suppose that \((v, w)\in A_{p,q}^\alpha\) and \(v^{-\frac{p'}{p}}, w\in A_{\infty}\). Then, it follows that
\[
\|J^{-\alpha}f\|_{L^q(w)}\lesssim \| f\|_{L^p(v)},\ \forall f\in \mcal{S}(\R^d).
\]
\end{thm}
We now turn to estimate \eqref{eq4.2}. The idea is to apply again Lerner's domination theorem. Of course we know that the operators \(J^{i\tau}\) are classical pseudodifferential operators of order zero and therefore they are Calderón-Zygmund operators. That means that \(J^{i\tau}\) can be controlled by a sparse operator. The problem here is that we need to understand precisely how the implicit constant depends on \(\tau\), so that we can be sure that the three lines lemma can be properly applied. For this reason we will approach the problem from a different direction. We will first prove an identity very similar to identity \eqref{eq2.6}, with precise information about the growth of the constants. 

Define the function
\[
A(z)=\int_0^{+\infty}x^{z-1}(e^{-x/2}-e^{-2x})dx,\ \mbox{Re}(z)\geq 0.
\]
Note that \(A(z)=(2^z-2^{-z})\Gamma(z)\). In particular it follows that \(A\) is zero at the points
\[
z=\frac{2\pi i k}{\log 4},\ k\in \Z\setminus \{0\}. 
\]
If we make the change of variables \(x=\eta t, \eta>0\), then
\[
A(z)=\eta^z\int_0^{+\infty}t^{z-1}(e^{-\eta t/2}-e^{-2\eta t})dt.
\]
Setting \(\eta=1+|\xi|^2\) and \(z=i\tau/2\) it follows that 
\[
\langle \xi\rangle^{i\tau}=\frac{1}{A(-i\tau/2)}\int_0^{+\infty}t^{-i\frac{\tau}{2}-1}f_t(\xi)dt,
\]
where
\[
f_t(\xi)=e^{-\frac{1}{2}(1+|\xi|^2)t}-e^{-2(1+|\xi|^2)t}
\]
and \(\tau\in \R\) is such that \(A(-i\tau/2)\neq 0\). That, is \(\tau\notin Z\) where \(Z=\{4\pi k/\log 4: k\in \Z\setminus\{0\}\}\). Note that we can explicitly compute the inverse Fourier transform of \(f_t\),
\[
(f_t)^\vee(x)=\sqrt{2\pi}^d\frac{e^{-t/2}}{\sqrt{t}^d}e^{-2\pi^2\frac{|x|^2}{t}}-\sqrt{\frac{\pi}{2}}^d\frac{e^{-2t}}{\sqrt{t}^d}e^{-\pi^2\frac{|x|^2}{2t}}.
\]
Now we use this to compute the Fourier transform of \(\langle \xi\rangle^{i\tau}\). Given \(\phi\in\mcal{S}(\R^d)\) and \(\tau\notin Z\) we have that
\[
\begin{split}
  \langle (\langle \cdot \rangle^{i\tau})^\vee,\phi\rangle&=\int \langle \xi\rangle^{i\tau}\check{\phi}(\xi)d\xi =  \int \frac{1}{A(-i\tau/2)}\int_0^{\infty}t^{-i\frac{\tau}{2}-1}f_t(\xi)\check{\phi}(\xi)dtd\xi\\
  &=\frac{1}{A(-i\tau/2)}\int_0^{\infty}t^{-i\frac{\tau}{2}-1}\int f_t(\xi)\check{\phi}(\xi)d\xi dt\\
  &=\frac{1}{A(-i\tau/2)}\int_0^{\infty}t^{-i\frac{\tau}{2}-1}\int \check{f}_t(x)\phi(x)dx dt\\
  &=\frac{1}{A(-i\tau/2)}\int_0^\infty t^{-i\frac{\tau}{2}-1}\int \phi(x)\left(\sqrt{2\pi}^d\frac{e^{-t/2}}{\sqrt{t}^d}e^{-2\pi^2\frac{|x|^2}{t}}-\sqrt{\frac{\pi}{2}}^d\frac{e^{-2t}}{\sqrt{t}^d}e^{-\pi^2\frac{|x|^2}{2t}}\right)dxdt.
\end{split}
\]
Now suppose that \(\phi(0)=0\). Then, \(|\phi(x)|\leq \|\nabla \phi\|_{L^\infty}|x|\). This estimate will allow us to switch the order of integration. To show this we must prove that
\[
\int_0^\infty\int_{\R^d}\left|t^{-i\frac{\tau}{2}-1}\phi(x)\left(\sqrt{2\pi}^d\frac{e^{-t/2}}{\sqrt{t}^d}e^{-2\pi^2\frac{|x|^2}{t}}-\sqrt{\frac{\pi}{2}}^d\frac{e^{-2t}}{\sqrt{t}^d}e^{-\pi^2\frac{|x|^2}{2t}}\right)\right|dxdt<\infty. 
\]
By the triangle inequality this double integral is bounded by
\[
\sqrt{2\pi}^d\int_0^\infty \int_{\R^d}t^{-1}|\phi(x)|\frac{e^{-t/2}}{\sqrt{t}^d}e^{-2\pi^2\frac{|x|^2}{t}}dxdt+\sqrt{\frac{\pi}{2}}^d\int_0^\infty\int_{\R^d}t^{-1}|\phi(x)|\frac{e^{-2t}}{\sqrt{t}^d}e^{-\pi^2\frac{|x|^2}{2t}}dxdt. 
\]
These two terms are a multiple of each other and so it suffices to check that the first one is finite. When \(t\) is small we use the estimate \(e^{-t/2}\leq 1\) and when \(t\) is large \(e^{-2\pi^2|x|^2/t}\leq 1\). Doing this the first term is bounded by
\[
\sqrt{2\pi}^d\int_0^1\int_{\R^d}|\phi(x)|t^{-\frac{d}{2}-1}e^{-2\pi^2\frac{|x|^2}{t}}dxdt+\sqrt{2\pi}^d\int_1^\infty\int_{\R^d}|\phi(x)|t^{-\frac{d}{2}-1}e^{-t/2}dxdt=:I_1+I_2.
\]
The term \(I_2\) is clearly finite, so we turn our attention to \(I_1\). We have that
\[
I_1=\sqrt{2\pi}^d\int_{\R^d}|\phi(x)|\int_0^1t^{-\frac{d}{2}-1}e^{-2\pi^2\frac{|x|^2}{t}}dtdx.
\]
Now consider the change of variables \(s=|x|^2/t\),
\[
\begin{split}
I_1&=\sqrt{2\pi}^d\int_{\R^d}|\phi(x)|\int_{|x|^2}^\infty \frac{s^{\frac{d}{2}}}{|x|^d}e^{-2\pi^2 s}dsdx\\
&\leq \int_{\R^d}\frac{|\phi(x)|}{|x|^d}dx\int_0^\infty s^{\frac{d}{2}}e^{-2\pi^2s}ds
\end{split}
\]
and this is finite since \(|\phi(x)|\lesssim |x|\). Therefore we may apply Fubini's theorem to conclude that
\[
\langle (\langle \cdot\rangle^{i\tau})^\vee,\phi\rangle=\frac{1}{A(-i\tau/2)}\int_{\R^d}\phi(x)\int_0^\infty t^{-i\frac{\tau}{2}-1}\left[\sqrt{2\pi}^d\frac{e^{-t/2}}{\sqrt{t}^d}e^{-2\pi^2\frac{|x|^2}{t}}-\sqrt{\frac{\pi}{2}}^d\frac{e^{-2t}}{\sqrt{t}^d}e^{-\pi^2\frac{|x|^2}{2t}}\right]dtdx.
\]
This expression can be simplified. By splitting into a sum of two integrals and making the changes of variables \(t=2s\) and \(t=s/2\) in each one we get
\[
\begin{split}
&\sqrt{2\pi}^d\int_0^\infty t^{-i\frac{\tau}{2}-1-\frac{d}{2}}e^{-t/2-2\pi^2|x|^2/t}dt-\sqrt{\frac{\pi}{2}}^d\int_0^\infty t^{-i\frac{\tau}{2}-1-\frac{d}{2}}e^{-2t-\pi^2\frac{|x|^2}{2t}}dt\\
&=(2^{-i\frac{\tau}{2}}-2^{i\frac{\tau}{2}})\pi^\frac{d}{2}\int_0^\infty s^{-i\frac{\tau}{2}-1-\frac{d}{2}}e^{-s-\pi^2\frac{|x|^2}{s}}ds.
\end{split}
\]
Therefore, noting that \(A(-i\tau/2)=(2^{-i\tau/2}-2^{i\tau/2})\Gamma(-i\tau/2)\) we get
\[
\langle (\langle \cdot\rangle^{i\tau})^\vee,\phi\rangle=\frac{\pi^\frac{d}{2}}{\Gamma(-i\tau/2)}\int_{\R^d}\phi(x)\int_0^\infty s^{-i\frac{\tau}{2}-1-\frac{d}{2}}e^{-s-\pi^2\frac{|x|^2}{s}}dsdx.
\]
Note that this equality is actually valid for all \(\tau\). Indeed, we have just shown it holds for \(\tau\notin Z\), but both sides are continuous in \(\tau\), so the equality remains true for \(\tau\in Z\) also. 

Put
\[
G_{i\tau}(x)=\frac{\pi^\frac{d}{2}}{\Gamma(-i\tau/2)}\int_0^\infty s^{-i\frac{\tau}{2}-1-\frac{d}{2}}e^{-s-\pi^2\frac{|x|^2}{s}}ds,
\]
and \(\tilde{W}_\tau=(\langle \cdot \rangle^{i\tau})^\vee\). We have shown the following lemma.
\begin{lemma}\label{lemma4.1}
For \(\phi\in \mcal{S}(\R^d)\) with \(\phi(0)=0\),
\[
\langle \tilde{W}_\tau,\phi\rangle=\int_{\R^d}\phi(x)G_{i\tau}(x)dx,\ \forall \tau\in \R. 
\]
\end{lemma}
Arguing as in the proof of proposition \ref{prop4.1} we can show the following.
\begin{prop}\label{prop4.3}
We have that
\[
|G_{i\tau}(x)|\lesssim_d\frac{1}{|\Gamma(-i\tau/2)|}\begin{dcases}
e^{-\pi|x|},&|x|\geq 1/\pi\\
|x|^{-d},&|x|\leq 1/\pi. 
\end{dcases}
\]
\end{prop}
\begin{proof}
We include the proof for completeness. First suppose that \(|x|\geq 1/\pi\). We have that
\[
s+\pi^2\frac{|x|^2}{s}\geq s+\frac{1}{s}
\]
and 
\[
s+\pi^2\frac{|x|^2}{s}\geq \min_{s>0}s+\pi^2\frac{|x|^2}{s}=2\pi |x|.
\]
Therefore,
\[
-s-\pi^2\frac{|x|^2}{s}\leq -\frac{s}{2}-\frac{1}{2s}-\pi|x|.
\]
Using this estimate we obtain
\[
|G_{i\tau}(x)|\leq \frac{\pi^{d/2}}{|\Gamma(-i\tau/2)|}\int_0^\infty s^{-1-\frac{d}{2}}e^{-s/2-\frac{1}{2s}}dx e^{-\pi|x|}\lesssim_d\frac{e^{-\pi|x|}}{|\Gamma(-i\tau/2)|}.
\]
Now suppose that \(|x|\leq 1/\pi\), and write \(G_{i\tau}=G^1+G^2+G^3\), where
\begin{align*}
G^1(x)&=\frac{\pi^{d/2}}{\Gamma(-i\tau/2)}\int_0^{|x|^2}s^{-i\frac{\tau}{2}-1-\frac{d}{2}}e^{-s-\pi^2\frac{|x|^2}{s}}ds\\
G^2(x)&=\frac{\pi^{d/2}}{\Gamma(-i\tau/2)}\int_{|x|^2}^{1/\pi^2}s^{-i\frac{\tau}{2}-1-\frac{d}{2}}e^{-s-\pi^2\frac{|x|^2}{s}}ds\\
G^3(x)&=\frac{\pi^{d/2}}{\Gamma(-i\tau/2)}\int_{1/\pi^2}^{+\infty}s^{-i\frac{\tau}{2}-1-\frac{d}{2}}e^{-s-\pi^2\frac{|x|^2}{s}}ds.
\end{align*}
Using the estimate \(e^{-s-\pi^2|x|^2/s}\leq e^{-\pi^2|x|^2/s}\) in \(G^1\), the estimate \(e^{-s-\pi^2|x|^2/s}\leq 1\) in \(G^2\) and \(e^{-s-\pi^2|x|^2/s}\leq e^{-s}\) in \(G^3\), the result follows.
\end{proof}
Using these estimates we can now prove the next lemma.
\begin{lemma}\label{lemma4.2}
For \(\tau\neq 0\) let \(\delta_n=e^{-2\pi \frac{n}{|\tau|}}\). Then the limit 
\[
\lim_n\int_{\delta_n\leq |x|\leq 1}G_{i\tau}(x)dx
\] 
exists.
\end{lemma}
\begin{proof}
Define \(g(s)=e^{-s}-1\) and note that \(|g(s)|\leq \min\{s,1\}, s\geq 0\). We write 
\[
G_{i\tau}(x)=\frac{\pi^{d/2}}{\Gamma(-i\tau/2)}\int_0^\infty s^{-i\frac{\tau}{2}-1-\frac{d}{2}}e^{-\pi^2\frac{|x|^2}{s}}ds+\frac{\pi^{d/2}}{\Gamma(-i\tau/2)}\int_0^\infty s^{-i\frac{\tau}{2}-1-\frac{d}{2}}g(s)e^{-\pi^2\frac{|x|^2}{s}}ds.
\]
Using the change of variables \(t=|x|^2/s\) in the first integral we get
\[
\frac{\pi^{d/2}}{\Gamma(-i\tau/2)}|x|^{-d-i\tau}\int_0^\infty t^{i\frac{\tau}{2}-1+\frac{d}{2}}e^{-\pi^2t}dt.
\]
Now note that
\[
\int_{\delta_n\leq |x|\leq 1}|x|^{-d-i\tau}dx=d|B_1|\int_{\delta_n}^1\rho^{-d-i\tau} \rho^{d-1}d\rho = \frac{d}{i\tau}|B_1|\left(\delta_n^{-i\tau}-1\right)=0,
\]
since \(\delta_n^{-i\tau}=e^{2\pi i n\frac{\tau}{|\tau|}}=1\). This means that 
\[
\int_{\delta_n\leq |x|\leq 1}G_{i\tau}(x)dx=\frac{\pi^{d/2}}{\Gamma(-i\tau/2)}\int_{\delta_n\leq |x|\leq 1}\int_0^\infty s^{-i\frac{\tau}{2}-1-\frac{d}{2}}g(s)e^{-\pi^2\frac{|x|^2}{s}}dsdx. 
\]
The result now follows from the dominated convergence theorem if only we can prove that the integrand can be dominated by an integrable function in \(\{|x|\leq 1\}\). To this end we estimate the integrand as follows:
\[
\begin{split}
  \left|\int_0^\infty s^{-i\frac{\tau}{2}-1-\frac{d}{2}}g(s)e^{-\pi^2\frac{|x|^2}{s}}ds\right|&\leq \int_0^1s^{-1-\frac{d}{2}}|g(s)|e^{-\pi^2\frac{|x|^2}{s}}ds+\int_1^\infty s^{-1-\frac{d}{2}}|g(s)|e^{-\pi^2\frac{|x|^2}{s}}ds\\
  &\leq \int_0^1s^{-\frac{d}{2}}e^{-\pi^2\frac{|x|^2}{s}}ds+\int_1^\infty s^{-1-\frac{d}{2}}e^{-\pi^2\frac{|x|^2}{s}}ds\\
  &=\frac{1}{|x|^{d-2}}\int_{|x|^2}^\infty t^{\frac{d}{2}-2}e^{-\pi^2t}dt+\frac{1}{|x|^d}\int_0^{|x|^2}t^{\frac{d}{2}-1}e^{-\pi^2t}dt\\
  &\leq \frac{1}{|x|^{d-2}}\int_{|x|^2}^1t^{\frac{d}{2}-2}dt+\frac{1}{|x|^{d-2}}\int_1^\infty t^{\frac{d}{2}-2}e^{-\pi^2t}dt\\
  &+\frac{1}{|x|^d}\int_0^{|x|^2}t^{\frac{d}{2}-1}dt.
\end{split}
\]
Here we must separate the cases \(d=2\) and \(d\neq 2\). Suppose first that \(d\neq 2\). Then the last expression is equal to
\[
\frac{2}{d-2}\frac{1}{|x|^{d-2}}-\frac{2}{d-2}+\frac{1}{|x|^{d-2}}\int_1^\infty t^{\frac{d}{2}-2}e^{-\pi^2t}dt+\frac{2}{d},
\]
which is integrable in \(\{|x|\leq 1\}\). If instead \(d=2\), then we get
\[
\int_1^\infty t^{-1}e^{-\pi^2t}dt+1-2\log|x|,
\]
which is also integrable. Therefore the dominated convergence theorem is applicable and 
\[
\lim_n\int_{\delta_n\leq |x|\leq 1}G_{i\tau}(x)dx=\frac{\pi^{d/2}}{\Gamma(-i\tau/2)}\int_{|x|\leq 1}\int_0^\infty s^{-i\frac{\tau}{2}-1-\frac{d}{2}}g(s)e^{-\pi^2\frac{|x|^2}{s}}dsdx.
\]
\end{proof}
Now, given \(\tau\neq 0\), define
\[
\langle \tilde{V}_\tau,\phi\rangle = \lim_n\int_{|x|\geq \delta_n}G_{i\tau}(x)\phi(x)dx,\ \forall \phi\in \mcal{S}(\R^d).
\]
We can write
\[
\begin{split}
\int_{|x|\geq \delta_n}G_{i\tau}(x)\phi(x)dx&=\int_{\delta_n\leq |x|\leq 1}G_{i\tau}(x)(\phi(x)-\phi(0))dx+\phi(0)\int_{\delta_n\leq |x|\leq 1}G_{i\tau}(x)dx\\
&+\int_{|x|\geq 1}G_{i\tau}(x)\phi(x)dx.
\end{split}
\]
By proposition \ref{prop4.3} and lemma \ref{lemma4.2}, it follows that the limit as \(n\rightarrow +\infty\) exits. Moreover, 
\[
\begin{split}
|\langle \tilde{V}_\tau,\phi\rangle|&\leq \int_{|x|\leq 1}|G_{i\tau}(x)||\phi(x)-\phi(0)|dx\\
&+|\phi(0)|\frac{\pi^{d/2}}{|\Gamma(-i\tau/2)|}\int_{|x|\leq 1}\int_0^\infty s^{-1-\frac{d}{2}}|g(s)|e^{-\pi^2\frac{|x|^2}{s}}dsdx+\int_{|x|\geq 1}|G_{i\tau}(x)||\phi(x)|dx\\
&\lesssim_{d,\tau} \|\nabla \phi\|_{L^\infty}\int_{|x|\leq 1}|x|^{1-d}dx+\|\phi\|_{L^\infty}+\|\phi\|_{L^\infty}\int_{|x|\geq 1}e^{-\pi|x|}dx.
\end{split}
\]
This shows that \(\tilde{V}_\tau\) is a well-defined tempered distribution. It is also clear that \(\tilde{V}_\tau=G_{i\tau}\) in \(\R^d\setminus \{0\}\). Next we check that \(G_{i\tau}\) is a standard kernel.
\begin{lemma}\label{lemma4.3}
The following estimates hold:
\begin{align*}
    |G_{i\tau}(x)|&\lesssim_d \sqrt{|\tau|}e^{\frac{\pi}{4}|\tau|}|x|^{-d}\\
    |G_{i\tau}(x-y)-G_{i\tau}(x)|&\lesssim_d \sqrt{|\tau|}e^{\frac{\pi}{4}|\tau|}\frac{|y|}{|x|^{d+1}},\text{ when }|x|\geq 2|y|.
\end{align*}
\end{lemma}
\begin{proof}
By proposition \ref{prop4.3} we already know that
\[
|G_{i\tau}(x)|\lesssim_d \frac{1}{|\Gamma(-i\tau/2)|}|x|^{-d}.
\]
The first estimate follows from the fact that 
\[
|\Gamma(-i\tau/2)|^2=\frac{2\pi}{\tau \sinh(\pi \tau/2)}.
\]
Now we turn our attention to the second estimate. We have that
\[
G_{i\tau}(x-y)-G_{i\tau}(x)=\frac{\pi^{d/2}}{\Gamma(-i\tau/2)}\int_0^\infty s^{-i\frac{\tau}{2}-1-\frac{d}{2}}e^{-s}(e^{-\frac{\pi^2}{s}|x-y|^2}-e^{-\frac{\pi^2}{s}|x|^2})ds.
\]
Define \(\Phi_s(x)=e^{-\pi^2|x|^2/s}\). By the fundamental theorem of Calculus,
\[
\begin{split}
\Phi_s(x-y)-\Phi_s(x)&=\int_0^1\frac{d}{dt}\left\{\Phi_s(x-ty)\right\}dt=-\int_0^1\nabla \Phi_s(s-ty)\cdot ydt\\
&=\frac{2\pi^2}{s}\int_0^1e^{-\frac{\pi^2}{s}|x-ty|^2}(x-ty)\cdot ydt.
\end{split}
\]
Note that, since \(|x|\geq 2|y|\) it follows that \(|x-ty|\geq |x|/2\). Therefore,
\[
|\Phi_s(x-y)-\Phi_s(x)|\leq \frac{2\pi^2}{s}\int_0^1e^{-\frac{\pi^2}{4s}|x|^2}|x-ty||y|dt\leq \frac{3\pi^2}{s}|x||y|e^{-\frac{\pi^2}{4s}|x|^2}.
\]
Using this estimate we obtain
\[
|G_{i\tau}(x-y)-G_{i\tau}(x)|\leq \frac{3\pi^{d/2+2}}{|\Gamma(-i\tau/2)|}|x||y|\int_0^\infty s^{-2-\frac{d}{2}}e^{-s-\frac{\pi^2}{4s}|x|^2}ds.
\]
If we now argue as in the proof of proposition \ref{prop4.3} we see that
\[
\int_0^\infty s^{-2-\frac{d}{2}}e^{-s-\frac{\pi^2}{4s}|x|^2}ds\lesssim_d \frac{1}{|x|^{d+2}}.
\]
Therefore the result follows.
\end{proof}
Now take \(\phi\in \mcal{S}(\R^d)\) and \(\rho\in \mcal{S}(\R^d)\) such that \(\rho(0)=1\). Then, by lemma \ref{lemma4.1},
\[
\begin{split}
  \langle \tilde{W}_\tau, \phi\rangle &= \langle \tilde{W}_\tau, \phi-\phi(0)\rho\rangle+\langle \tilde{W}_\tau, \rho\rangle \phi(0)=\int_{\R^d}G_{i\tau}(x)[\phi(x)-\phi(0)\rho(x)]dx+\langle \tilde{W}_\tau, \rho\rangle \phi(0)  \\
  &=\lim_n\int_{|x|\geq \delta_n}G_{i\tau}(x)\phi(x)dx-\phi(0)\lim_n\int_{|x|\geq \delta_n}G_{i\tau}(x)\rho(x)dx+\langle \tilde{W}_\tau,\rho\rangle \phi(0)\\
  &=\langle \tilde{V}_\tau,\phi\rangle+\phi(0)\langle \tilde{W}_\tau-\tilde{V}_\tau,\rho\rangle.
\end{split}
\]
Thus,
\[
\tilde{W}_\tau=\tilde{V}_\tau+ \tilde{\beta}_d(\tau)\delta,
\]
where
\[
\tilde{\beta}_d(\tau)=\langle \tilde{W}_\tau-\tilde{V}_\tau, \rho\rangle. 
\]
Convolving with a function \(f\in \mcal{S}(\R^d)\) we arrive at our desired identity:
\begin{equation}\label{eq4.3}
    J^{i\tau}f=\tilde{V}_\tau * f+\tilde{\beta}_d(\tau)f,\ \forall \tau \neq 0.
\end{equation}
Put \(\tilde{T}_\tau f= \tilde{V}_\tau * f\). From identity \eqref{eq4.3} it follows immediately that 
\[
\|\tilde{T}_\tau f\|_{L^2}\leq \|J^{i\tau}f\|_{L^2}+|\tilde{\beta}_d(\tau)|\|f\|_{L^2}\leq(1+|\langle \langle \cdot \rangle^{i\tau}, \check{\rho}\rangle | +|\langle \tilde{V}_\tau, \rho\rangle |)\|f\|_{L^2} \lesssim_d(1+\sqrt{|\tau|}e^{\frac{\pi}{4}|\tau|})\|f\|_{L^2}.
\]
So, the operator \(\tilde{T}_\tau\) has a bounded extension to \(L^2\). This means that \(\tilde{T}_\tau\) is bounded on \(L^2\) and is associated to a standard kernel. Thus, \(\tilde{T}_\tau\) is a Calderón-Zygmund operator. This allows us to conclude that \(\tilde{T}_\tau\) has a unique extension to an operator bounded from \(L^1\) to \(L^{1,\infty}\), with norm
\[
\|\tilde{T}_\tau\|_{L^1\rightarrow L^{1,\infty}}\lesssim_d 1+\sqrt{|\tau|}e^{\frac{\pi}{4}|\tau|}.
\]
This in turn implies that \(J^{i\tau}\) also has a bounded extension mapping \(L^1\) to \(L^{1,\infty}\) with
\[
\|J^{i\tau}\|_{L^1\rightarrow L^{1,\infty}}\lesssim_d 1+ \sqrt{|\tau|}e^{\frac{\pi}{4}|\tau|}.
\]
Note that this estimate works for all \(\tau\in \R\) since \(J^{i\tau}\) is just the identity when \(\tau=0\). To apply Lerner's domination theorem we have only to check that the associated maximal operator
\[
\tilde{M}_\tau f(x)=\sup_{Q\ni x}\sup_{x',x''\in Q}|J^{i\tau}(f\chi_{\R^d\setminus Q^*})(x')-J^{i\tau}(f\chi_{\R^d\setminus Q^*})(x'')|
\]
is weak type \((1,1)\). Again, this follows from identity \eqref{eq4.3} and the fact that \(\tilde{T}_\tau\) is a Calderón-Zygmund operator. Arguing as in lemma \ref{lemma2.6} we get
\[
\|\tilde{M}_\tau\|_{L^1\rightarrow L^{1,\infty}}\lesssim_d\sqrt{|\tau|}e^{\frac{\pi}{4}|\tau|}.
\]
We can thus apply theorem \ref{thm1.2}, theorem \ref{bound for sparse operators} and the argument in the proof of theorem \ref{thm2.3} to conclude the following result.
\begin{thm}\label{thm4.2}
Let \(1<p<\infty\) and consider weights \((v,w)\in A_{p,p}\) such that \(v^{-p'/p}, w\in A_\infty\). Then,
\[
\|J^{i\tau}f\|_{L^p(w)}\lesssim_{d,p,v,w}(1+\sqrt{|\tau|} e^{\frac{\pi}{4}|\tau|})\|f\|_{L^p(v)},\ \forall f\in \mcal{S}(\R^d),\forall \tau\in \R.  
\]
\end{thm}
Using theorems \ref{thm4.1} and \ref{thm4.2} we can prove all the main results of the previous sections for the operator \(J\) with the added benefit that no polynomials are necessary and no difficulty is presented by the case \(t=0\). Indeed, the analogous of theorem \ref{thm2.2} is the following Sobolev type inequality.
\begin{thm}\label{thm4.last}
Let \(1<p\leq q<\infty\) and \(t<s<t+d\). Suppose that \((v,w)\in A_{p,q}^{s-t}\) and \(v^{-p'/p},w\in A_\infty\). Then,
\[
\|J^tf\|_{L^q(w)}\lesssim \|J^sf\|_{L^p(v)},\ \forall f\in \mcal{S}(\R^d).
\]
\end{thm}
The statements of theorems \ref{thm2.4}, \ref{thm3.1} and \ref{thm3.2} are exactly the same for \(J\) as they are for \(D\).

\section{Application to a mixed inequality}\label{mixed}

Consider a measurable function \(f:\R^d\times \R^m\rightarrow \C\). The mixed \(L^p_xL^q_y\) norm of \(f\) is simply
\[
\|f\|_{L^p_xL^q_y}=\left( \int (\int |f(x,y)|^q dy)^{\frac pq} dx\right) ^{\frac 1p},
\]
when $1\leq p,q < \infty$, with the obvious adaptation for $L^\infty$. In general, if \(1\leq p\leq q\leq\infty\), the inequality 
$
\|f\|_{L^q_yL^p_x}\leq \|f\|_{L^p_xL^q_y}
$
holds, and it can be proved directly by using Minkowski's integral inequality. 

Here we are interested in the reverse inequality, which is not as straightforward. Typically, this
is achieved by a combination of the use of weights and Hölder's inequality, when one wants to raise the Lebesgue exponents, with Sobolev's inequality to lower them, by paying the price of having to introduce derivatives. For example, in this situation of mixed $L^p$ norms, if one wants to bound the $L^p_xL^q_y$  norm of $f$ by $L^q_yL^p_x$, one could start by using the Sobolev inequality in the $y$ variable to lower the $q$ exponent to an $L^p_y$ norm, at which point the order of the $x$ and $y$ norms, both in $L^p$, can be switched, and then finally the exponent of the $y$ norm raised 
again to $L^q_y$ through the introduction of a convenient weight with Hölder's inequality. We would thus end up controlling the mixed 
$L^p_xL^q_y$ norm of \(f\) by the mixed $L^q_yL^p_x$ norm of \(D^s_yf\) with a weight also in $y$. However, for certain applications, we might be interested in controlling \(f\) by derivatives and weights specifically in the $x$ variable. In order to use the same type of argument, 
one has to reverse the sequence of steps, by first using a weight with Hölder in the $x$ variable, to be able to raise the $L^p_x$ norm to $L^q_x$, then switch the $x$ and $y$ norms,  both in $L^q$, and only at the last step
apply the Sobolev inequality in $x$,  to bring the $L^q_x$ norm back down to $L^p_x$.
However,  a weighted
Sobolev norm in the $x$ variable is now required, because the weight in $x$ is already there from the first step of the procedure. Corollary \ref{cor2.1} is exactly the type of weighted Sobolev inequality that allows us to do it.

 The following theorem, therefore, yields the result that we have just described.
\begin{thm}\label{thm5.1}
Let \(1<p\leq q<\infty\), \(0<s<d\), and suppose we have two weights \(v,w\) such that \((v,w)\in A_{p,q}^s\), \(v^{-p'/p}, w\in A_\infty\) and
\[
\int_{\R^d}w(x)^{-\frac{p}{q-p}}dx<\infty;\ w(\Gamma^S_R)=+\infty, \forall R>0\text{ and }S\subseteq S^{d-1}, |S|>0; L^p(v)\subseteq L^p.
\]
If \(f:\R^d\times \R^m\rightarrow \C\) is a measurable function such that \(f(\cdot,y)\in D^{\geq0}\mcal{S}(\R^d)\) for a.e. \(y\in \R^m\), then 
\[
\|f(x,y)\|_{L^p_xL^q_y}\lesssim\|v(x)^{1/p}D_x^sf(x,y)\|_{L^q_yL^p_x}.
\]
\end{thm}
\begin{proof}
By corollary \ref{cor2.1} and remark \ref{remark1} we have that 
\[
\|f(\cdot, y)\|_{L^q(w)}\lesssim\|D^s_xf(\cdot, y)\|_{L^p(v)}\text{ for a.e. }y\in \R^m.
\]
Using this and Hölder's inequality we see that
\[
\begin{split}
  \|f\|_{L^p_xL^q_y}&\leq \|\|f(\cdot, \cdot)\|_{L^q_y}\|_{L^q_x(w)}\left(\int w(x)^{-\frac{p}{q-p}}dx\right)^{\frac{1}{p}-\frac{1}{q}} \lesssim_w\|\|f(\cdot,\cdot)\|_{L^q_x(w)}\|_{L^q_y}\\
  &\lesssim\|\|D^s_xf(\cdot,\cdot)\|_{L^p_x(v)}\|_{L^q_y}\lesssim\|v(x)^{1/p}D^s_xf(x,y)\|_{L^q_yL^p_x}.
\end{split}
\]
\end{proof}
\begin{remark}
Observe that, even though only one weight, $v$, appears in the main inequality, a suitable pairing $(v,w)$ is required for the conditions 
of the theorem to hold. 
\end{remark}
If we choose non-homogeneous weights we get the following result.
\begin{cor}
Let \(1<q<\infty\),  $2q/(q+1)<p<q,\, s=d(1/p-1/q)$ and let \(\gamma\in]s, d/p'[\). If \(f\) is a measurable function such that \(f(\cdot,y)\in D^{\geq0}\mcal{S}(\R^d)\) for a.e. \(y\in \R^m\), then
\[
\|f(x,y)\|_{L^p_xL^q_y}\lesssim\|\langle x\rangle^\gamma D^s_xf(x,y)\|_{L^q_yL^p_x}.
\]
\end{cor}
\begin{proof}
Apply the previous theorem with \(v(x)=\langle x\rangle^{\gamma p} \) and \(w(x)=\langle x\rangle^{\gamma q}\).
\end{proof}
\begin{remark}
The condition \(p>2q/(q+1)\) ensures that \(s<d/p'\).
\end{remark}
Of course these arguments can be adapted to non-homogeneous derivatives.
\begin{thm}
Let \(1<p\leq q<\infty\), \(0<s<d\), and suppose we have two weights \(v,w\) such that \((v,w)\in A_{p,q}^s, v^{-p'/p},w\in A_\infty\) and 
\[
\int_{\R^d}w(x)^{-\frac{p}{q-p}}dx<\infty. 
\]
If \(f:\R^d\times \R^m\rightarrow \C\) is a measurable function such that \(f(\cdot,y)\in \mcal{S}(\R^d)\) for a.e. \(y\in \R^m\), then
\[
\|f(x,y)\|_{L^p_xL^q_y}\lesssim \|v(x)^{1/p}J^s_xf(x,y)\|_{L^q_yL^p_x}.
\]
\end{thm}
\begin{proof}
The proof is the same as the proof of theorem \ref{thm5.1} except we use theorem \ref{thm4.last}.
\end{proof}
Applying this to non-homogeneous weights we get the following result.
\begin{cor}
Let \(1<q<\infty\),  $2q/(q+1)<p<q,\, s=d(1/p-1/q)$ and let \(\gamma\in]s, d/p'[\). If \(f\) is a measurable function such that \(f(\cdot,y)\in \mcal{S}(\R^d)\) for a.e. \(y\in \R^m\), then
\[
\|f(x,y)\|_{L^p_xL^q_y}\lesssim\|\langle x\rangle^\gamma J^s_xf(x,y)\|_{L^q_yL^p_x}.
\]
\end{cor}
\section{Acknowledgments}
Rodrigo Duarte would like to thank Javier Orts, Arber Selimi, Nino Scalbi and Frederico Toulson for the many helpful discussions. Jorge Drumond Silva would like to thank Felipe Linares and Ademir Pastor with whom technical discussions for nonlinear dispersive PDEs led to the subject of this work.

\appendix
\section{}\label{appendixA}
\begin{proof}[Proof of lemma \ref{extralemma1}]
Given our assumptions on \(f\) we may integrate \(\widehat{\partial^\gamma f}\) by parts sufficiently many times to get \(\widehat{\partial^\gamma f}(\xi)=(2\pi i \xi)^\gamma \hat{f}(\xi), |\gamma|\leq N\). Now fix \(\gamma\) with \(|\gamma|=N, \xi\neq 0\) and pick \(j\) such that \(|\xi_j|=\max_{1\leq k\leq d}|\xi_k|\). Then, by a change of variables we see that 
\[
\int_{\R^d}e^{-2\pi i \xi\cdot x}\partial^\gamma f(x)dx=-\int_{\R^d}e^{-2\pi i\xi\cdot x}\partial^\gamma f\left(x-\frac{e_j}{2\xi_j}\right)dx. 
\]
This shows that 
\[
\begin{split}
|\widehat{\partial^\gamma f}(\xi)|&=\left|\frac{1}{2}\int_{\R^d}e^{-2\pi i \xi\cdot x}\left[\partial^\gamma f(x)-\partial^\gamma f\left(x-\frac{e_j}{2\xi_j}\right)\right]dx\right|\\
&\leq \frac{C_\gamma}{2}\int_{\R^d}\frac{(2|\xi_j|)^{-\delta}}{(1+\min\{|x|,\left|x-\frac{e_j}{2\xi_j}\right|\})^{d+1}}dx.
\end{split}
\]
Note that by splitting the integral into the regions \(|x|\leq |x-e_j/(2\xi_j)|\) and \(|x|>|x-e_j/(2\xi_j)|\), and changing variables in the second region we obtain
\[
\int_{\R^d}\frac{1}{(1+\min\{|x|,\left|x-\frac{e_j}{2\xi_j}\right|\})^{d+1}}\leq 2\int_{\R^d}\frac{1}{(1+|x|)^{d+1}}dx<\infty. 
\]
Note also that \(|\xi|\leq \sqrt{d}|\xi_j|\), so \(|\xi|^\delta |\widehat{\partial^\gamma f}(\xi)|\lesssim_{\gamma, \delta,d} 1\). Now, since the function \(\xi\in S^{d-1}\mapsto \sum_{|\beta|=N}|\xi^\beta|\) has a positive minimum it follows that
\[
|\xi|^N\lesssim_{d,N}\sum_{|\beta|=N}|\xi^\beta|.
\]
Therefore,
\[
|\xi|^{N+\delta}|\hat{f}(\xi)|\lesssim_{d,N}|\xi|^\delta \sum_{|\beta|=N}|\xi^\beta \hat{f}(\xi)|\lesssim_{d,N} \sum_{|\beta|=N}|\xi|^\delta|\widehat{\partial^\beta f}(\xi)|\lesssim_{\delta,d,N}1.
\]
\end{proof}
\begin{lemma}\label{lemmaA.2}
If \(0<\delta\leq 1\) and \(\gamma\) is any multi-index, then the function \(\xi^\gamma|\xi|^{\delta-|\gamma|}\) is \(\delta\)-Hölder continuous.
\end{lemma}
\begin{proof}
First we consider the case \(\gamma=0\). Since \(0<\delta\leq 1\), then \(|\xi+\eta|^\delta\leq (|\xi|+|\eta|)^\delta\leq |\xi|^\delta+|\eta|^\delta, \forall \xi,\eta\). So, we have the two inequalities
\begin{align*}
    |\xi|^\delta&\leq |\xi-\eta|^\delta+|\eta|^\delta\\
    |\eta|^\delta&\leq |\eta-\xi|^\delta+|\xi|^\delta,
\end{align*}
so \(||\xi|^\delta-|\eta|^\delta|\leq |\xi-\eta|^\delta\). 

Now we consider the general case. Put \(f(\xi)=\xi^\gamma |\xi|^{\delta-|\gamma|}\). If \(\xi=0\) or \(\eta=0\) it is clear that \(|f(\xi)-f(\eta)|\leq |\xi-\eta|^\delta\). So, now suppose \(\xi,\eta\neq 0\). Then we may write
\[
f(\xi)-f(\eta)=\tilde{\eta}^\gamma[|\xi|^\delta-|\eta|^\delta]-|\xi|^\delta[\tilde{\eta}^\gamma-\tilde{\xi}^\gamma], 
\]
where \(\tilde{\xi}=\xi/|\xi|\). Interchanging \(\xi\) and \(\eta\) we obtain
\[
|f(\xi)-f(\eta)|\leq |\xi-\eta|^\delta+\min\{|\xi|,|\eta|\}^\delta|\tilde{\xi}^\gamma-\tilde{\eta}^\gamma|.
\]
Now consider the function \(g(x)=x^\gamma, |x|\leq 1\). This is a smooth function with gradient equal to \(\nabla g(x)=(\gamma_1x^{\gamma-e_1},\dots, \gamma_dx^{\gamma-e_d})\), which is in \(L^\infty(\overline{B}_1)\). By the mean value theorem it follows that \(|g(x)-g(y)|\lesssim_\gamma |x-y|\). In particular,
\[
\frac{|g(x)-g(y)|}{|x-y|^\delta}\lesssim_\gamma |x-y|^{1-\delta}\lesssim_{\gamma,\delta}1, \forall x,y\in\overline{B}_1, x\neq y.
\]
This implies that \(|\tilde{\xi}^\gamma-\tilde{\eta}^\gamma|\lesssim_{\gamma,\delta}|\tilde{\xi}-\tilde{\eta}|^\delta\). This way we obtain the estimate
\[
|f(\xi)-f(\eta)|\lesssim_{\gamma,\delta}|\xi-\eta|^\delta+\min\{|\xi|,|\eta|\}^\delta|\tilde{\xi}-\tilde{\eta}|^\delta. 
\]
Without loss of generality we may assume that \(|\xi|\leq |\eta|\). Put \(a=|\xi|\tilde{\eta}, b=\xi\) and \(\alpha=|\eta|/|\xi|\). Note that \(|a|=|b|, \alpha\geq 1\), and therefore
\[
|\alpha a-b|^2-|a-b|^2=(\alpha-1)[(\alpha+1)|a|^2-2a\cdot b]\geq 0
\]
because \(2a\cdot b\leq 2|a||b|\leq (\alpha+1)|a|^2\). This shows that 
\[
|\xi||\tilde{\xi}-\tilde{\eta}|=|a-b|\leq |\alpha a-b|=|\eta-\xi|.
\]
Thus,
\[
|f(\xi)-f(\eta)|\lesssim_{\gamma,\delta}|\xi-\eta|^\delta.
\]
\end{proof}
\begin{lemma}\label{lemmaA.3}
Consider \(0<\delta\leq 1, m\in \N_0, \gamma\) a multi-index and \(g\in \mcal{S}(\R^d)\). Then the function \(f(\xi)=\xi^\gamma|\xi|^{\delta+m-|\gamma|}g(\xi)\) satisfies the estimate
\[
|f(\xi)-f(\eta)|\leq C\frac{|\xi-\eta|^\delta}{(1+\min\{|\xi|,|\eta|\})^{d+1}}.
\]
\end{lemma}
\begin{proof}
We consider first the case \(m=0\). Using lemma \ref{lemmaA.2} we have that
\[
\begin{split}
|f(\xi)-f(\eta)|&\leq |g(\xi)||\xi^\gamma |\xi|^{\delta-|\gamma|}-\eta^\gamma|\eta|^{\delta-|\gamma|}|+|\eta^\gamma |\eta|^{\delta-|\gamma|}||g(\xi)-g(\eta)|\\
&\lesssim \langle \xi\rangle^{-M}|\xi-\eta|^\delta+|\eta|^\delta|g(\xi)-g(\eta)|, \text{ for any }M>0.
\end{split}
\]
Interchanging \(\xi\) and \(\eta\) we obtain the estimate
\[
|f(\xi)-f(\eta)|\lesssim (\max\{\langle \xi\rangle,\langle \eta\rangle\})^{-M}|\xi-\eta|^\delta+\min\{|\xi|,|\eta|\}^\delta|g(\xi)-g(\eta)|.
\]
Without loss of generality we may assume that \(|\xi|\leq |\eta|\). From 
\[
g(\xi)-g(\eta)=\int_0^1\nabla g(t\xi+(1-t)\eta)\cdot (\xi-\eta)dt,
\]
we get the estimate
\[
|\xi|^\delta |g(\xi)-g(\eta)|\leq |\xi|^\delta |\xi-\eta|\int_0^1|\nabla g(t\xi+(1-t)\eta)|dt\lesssim |\xi|^\delta |\xi-\eta|\int_0^1\langle t\xi+(1-t)\eta\rangle^{-M}dt.
\]
Now we assume that \(|\xi-\eta|<1\) and \(|\xi|\geq 2\). In this case, since \(\eta+t(\xi-\eta)\in B_1(\xi)\), we have that 
\[
\frac{1}{2}|\xi|\leq |\eta+t(\xi-\eta)|.
\]
Therefore we obtain
\[
|\xi|^\delta |g(\xi)-g(\eta)|\lesssim |\xi|^\delta|\xi-\eta|\langle \xi\rangle^{-M}\lesssim \frac{|\xi-\eta|^\delta}{\langle \xi\rangle^{M-\delta}}.
\]
Now suppose that \(|\xi-\eta|<1\) and \(|\xi|<2\). Then,
\[
|\xi|^\delta |g(\xi)-g(\eta)|\lesssim \|\nabla g\|_{L^\infty}|\xi|^\delta |\xi-\eta|\lesssim \frac{|\xi-\eta|^\delta}{(1+|\xi|)^{d+1}}.
\]
Finally let's suppose instead that \(|\xi-\eta|\geq 1\). Then we have that
\[
|\xi|^\delta|g(\xi)-g(\eta)|\lesssim |\xi|^\delta\langle \xi\rangle^{-M}+|\xi|^\delta\langle \eta\rangle^{-M}\lesssim |\xi|^\delta\langle \xi\rangle^{-M}\lesssim \frac{|\xi-\eta|^\delta}{\langle \xi\rangle^{M-\delta}}.
\]
Putting everything together and choosing \(M=d+1+\delta\) we find that 
\[
|f(\xi)-f(\eta)|\lesssim \frac{|\xi-\eta|^\delta}{(1+\min\{|\xi|,|\eta|\})^{d+1}}.
\]
Now suppose that \(m\geq 1\). In this case, \(f\) is at least \(C^1\). Again we use the estimate
\[
|f(\xi)-f(\eta)|\lesssim (\max\{\langle \xi\rangle, \langle \eta\rangle\})^{-M}|\xi^\gamma|\xi|^{\delta+m-|\gamma|}-\eta^\gamma|\eta|^{\delta+m-|\gamma|}|+\min\{|\xi|,|\eta|\}^{\delta+m}|g(\xi)-g(\eta)|.
\]
Again, without loss of generality we may assume that \(|\xi|\leq |\eta|\). The second term can be estimated in exactly the same way as the argument given in the case \(m=0\) using \(M=d+1+m+\delta\). As for the first term, we consider separately the cases \(|\xi-\eta|<1\) and \(|\xi-\eta|\geq 1\). When \(|\xi-\eta|\geq 1\) and choosing \(M=d+1+\delta+m\) we have that 
\[
\langle \eta\rangle^{-M}|\xi^\gamma|\xi|^{\delta+m-|\gamma|}-\eta^\gamma|\eta|^{\delta+m-|\gamma|}|\leq\langle \eta\rangle^{-M}(|\xi|^{\delta+m}+|\eta|^{\delta+m}|\lesssim \langle \xi\rangle^{-d-1}|\xi-\eta|^\delta.
\]
Now suppose that \(|\xi-\eta|<1\) and put \(h(\xi)=\xi^\gamma|\xi|^{\delta+m-|\gamma|}\). The function \(h\) is at least \(C^1\) and we have the estimate \(|\nabla h(\xi)|\lesssim |\xi|^{\delta+m-1}\). So, it follows that
\[
\begin{split}
|h(\xi)-h(\eta)|&\lesssim \int_0^1((1-t)|\eta|+t|\xi|)^{\delta+m-1}dt|\xi-\eta|\lesssim |\eta|^{\delta+m-1}|\xi-\eta|. 
\end{split}
\]
Therefore we see that
\[
\langle \eta\rangle^{-M}|h(\xi)-h(\eta)|\lesssim \langle \eta\rangle^{-M+\delta+m-1}|\xi-\eta|^\delta \lesssim \langle \xi\rangle^{-d-1}|\xi-\eta|^\delta.
\]
\end{proof}

\end{document}